\documentclass[11pt]{amsart}

\usepackage{mathrsfs}
\newtheorem{theorem}{Theorem}
\newtheorem{lemma}{Lemma}
\newtheorem{corollary}{Corollary}
\newtheorem{proposition}{Proposition}
\newtheorem{definition}{Definition}
\newtheorem{remark}{Remark}
\newtheorem{convention}{Convention}
\begin{document}
\setlength{\textwidth}{4.7in}
\setlength{\textheight}{7.5in}
\title[Isoparametric hypersurfaces]{Isoparametric hypersurfaces with four
principal curvatures, III}
\author{Quo-Shin Chi}

\thanks{The author was partially supported by NSF Grant No. DMS-0103838}
\address{Department of Mathematics, Washington University, St. Louis, MO 63130}
\email{chi@math.wustl.edu}
%%\date{}

\begin{abstract} The classification work~\cite{CCJ},~\cite{Chiq} left unsettled
only those anomalous 
isoparametric hypersurfaces with four principal curvatures and multiplicity
pair $\{4,5\},\{6,9\}$ or $\{7,8\}$ in the sphere.

By systematically exploring the ideal theory in commutative
algebra in conjunction with the geometry of isoparametric hypersurfaces,
we show that an isoparametric hypersurface with four principal
curvatures and multiplicities $\{4,5\}$ in $S^{19}$ is homogeneous, and, moreover,
an isoparametric hypersurface with four principal curvatures and
multiplicities $\{6,9\}$ in $S^{31}$ is
either the inhomogeneous one constructed by Ferus, Karcher and M\"{u}nzner, or the one that
is homogeneous.

This classification reveals the striking resemblance between these two rather different
types of isoparametric hypersurfaces in the homogeneous category, even though
the one with multiplicities $\{6,9\}$ is of the type constructed by Ferus, Karcher
and M\"{u}nzner and the one with multiplicities $\{4,5\}$ stands alone by
itself. The quaternion and the octonion algebras play a fundamental role in their
geometric structures.

A unifying theme in~\cite{CCJ},~\cite{Chiq} and the present sequel to them is Serre's criterion of normal 
varieties. Its technical side pertinent to our situation that we
developed in~\cite{CCJ},~\cite{Chiq} and extend in this sequel is instrumental.
  
The classification leaves only the case of multiplicity pair $\{7,8\}$ open.
\end{abstract}
\maketitle
\pagestyle{myheadings}
\markboth{QUO-SHIN CHI}{ISOPARAMETRIC HYPERSURFACES}
%*************
%%\footnote{Keywords: Isoparametric hypersurfaces; Subjclass: Primary, 53C40}

\section{Introduction} An isoparametric hypersurface $M$ in the sphere is one
whose principal curvatures and their multiplicities are fixed constants.
The classification of such hypersurfaces has been an outstanding problem in
submanifold geometry, listed as Problem 34 in~\cite{Ya}, as can be witnessed by its long history.
Through M\"{u}nzner's work~\cite{Mu}, we know the number $g$
of principal curvatures is 1,2,3,4 or 6, and there are at most two multiplicities
$\{m_1,m_2\}$ of the principal curvatures, occurring alternately when
the principal curvatures are ordered,
associated with $M$ ($m_1=m_2$ if $g$ is odd).
Over the ambient Euclidean space in which $M$ sits there is
a homogeneous polynomial $F$, called the Cartan-M\"{u}nzner polynomial, of degree $g$ that satisfies
\begin{equation}\nonumber
|\nabla F|^2(x)=g^2|x|^{2g-2},\quad (\Delta F)(x)=(m_2-m_1)g^2|x|^{g-2}/2
\end{equation}
whose restriction $f$ to the sphere has image in $[-1,1]$ with $\pm 1$ the only critical values.
For any $c\in(-1,1)$, the preimage $f^{-1}(c)$ is an isoparametric
hypersurface with $f^{-1}(0)=M$. This 1-parameter of
isoparametric hypersurfaces degenerates to the two submanifolds
$f^{-1}(\pm 1)$ of codimension $m_1+1$ and $m_2+1$ in the sphere.

The isoparametric hypersurfaces with $g=1,2,3$ were classified by
Cartan to be homogeneous
~\cite{Car2},~\cite{Car3}. For $g=6$, it is known that $m_1=m_2=1$
or 2 by Abresch~\cite{A}. Dorfmeister
and Neher~\cite{DN} showed that the isoparametric hypersurface is homogeneous in the
former case and Miyaoka~\cite{Mi2} settled the latter.

For $g=4$, there are infinite classes of inhomogeneous examples
of isoparametric hypersurfaces, two of which were first constructed by Ozeki and Tackeuchi~\cite[I]{OT} 
to be generalized later by Ferus, Karcher and M\"{u}nzner~\cite{FKM},
referred to collectively as isoparametric hypersurfaces of OT-FKM type subsequently. We remark that
the OT-FKM type includes all the homogeneous examples barring the two
with multiplicities $\{2,2\}$ and $\{4,5\}$. To construct the OT-FKM type,
let $P_{0},\cdots,P_{m_1}$ be a Clifford system on
${\mathbb R}^{2l}$, which are orthogonal symmetric operators on
${\mathbb R}^{2l}$ satisfying
$$
P_{i}P_{j}+P_{j}P_{i}=2\delta_{ij}I,\;\;i,j=0,\cdots,m.
$$
The ${\rm 4}$th degree homogeneous
polynomial
$$
F(x)=|x|^{4}-2\sum_{i=0}^{m}(<P_{i}(x),x>)^2
$$
is the Cartan-M\"{u}nzner polynomial. The two multiplicities of the OT-FKM type
are $m$ and $k\delta(m)-1$ for any $k=1,2,3,\cdots$, where $\delta(m)$ is the
dimension of an irreducible module of the Clifford algebra $C_{m-1}$
($l=k\delta(m)$). Stolz~\cite{St}
showed that these multiplicity pairs and $\{2,2\}$ and $\{4,5\}$
are exactly the possible multiplicities of isoparametric hypersurfaces
with four principal curvatures in the sphere.

The recent study of $n$-Sasakian manifolds~\cite{De}, Hamiltonian stability
of the Gauss images
of isoparametric hypersurfaces
in complex hyperquadrics as Lagrangian submanifolds~\cite{MO},~\cite{MO1}, isoparametric functions on
exotic spheres~\cite{GT}, and the realization
of the Cartan-M\"{u}nzner polynomial of an isoparametric hypersurface with
four principal curvatures as the moment map of a Spin-action on the ambient Euclidean
space, regarded as a cotangent bundle with the standard symplectic
structure~\cite{Fu},~\cite{Mi}, represent several new directions in
the study of such hypersurfaces.

Through~\cite{CCJ} (see also~\cite{Ch},~\cite{Ch1}) and~\cite{Chiq}
it has been clear by now that isoparametric hypersurfaces with four principal
curvatures and multiplicities $\{m_1,m_2\},m_1\leq m_2,$ fall into two
categories. Namely, the general category where $m_2\geq 2m_1-1,$ and the
anomalous category
where the multiplicities are $\{2,2\},\{3,4\},\{4,5\},\{6,9\}$ or $\{7,8\}$.
The former category
enjoys a rich connection with the theory of {\em reduced} ideals in commutative
algebra, and are exactly of OT-FKM type~\cite{CCJ},~\cite{Chiq}. The latter is peculiar, in that all known examples of
such hypersurfaces with multiplicities $\{3,4\},\{6,9\},$ or $\{7,8\}$
are of the OT-FKM type and have the property that incongruent isoparametric hypersurfaces with
the same multiplicity pair occur in the same ambient sphere, which is not
the case in
the former category; in contrast, those with multiplicities $\{2,2\}$ or $\{4,5\}$
can never be of OT-FKM type. The theory of reduced ideals breaks down in the anomalous
category. Yet, in~\cite{Chiq}, we were still able to utilize more
commutative algebra, in connection with the notion of Condition A introduced by
Ozeki and Takeuchi~\cite[I]{OT}, to prove that those hypersurfaces with
multiplicities $\{3,4\}$ are of OT-FKM type. This left unsettled only the
anomalous isoparametric
hypersurfaces with multiplicities $\{4,5\},\{6,9\}$ or $\{7,8\}$.

Of all known examples of isoparametric hypersurfaces
with four principal curvatures in the sphere, the homogeneous one ($=SU(5)/Spin(4)$)
with
multiplicities $\{4,5\}$ in $S^{19}$ is perhaps one of the most intriguing. First off
it stands alone by itself (together with the (classified) one with multiplicities $\{2,2\}$)
as it does not belong to the OT-FKM type.
More remarkably, through the work in~\cite{De}, one knows that there
is a contact CR structure of dimension 8 on its focal manifold of
dimension 14 in $S^{19}$, giving rise to the notion of 13-dimensional 5-Sasakian manifolds
fibered over ${\mathbb C}P^4$
that generalizes the 3-Sasakian ones. The 5-Sasakian manifold constructed
from the focal manifold carries a metric of positive sectional curvature~\cite{Be}.

Intuitively, it seems remote that the homogeneous example
$Spin(10)\cdot T^1/SU(4)\cdot T^1$
of multiplicities $\{6,9\}$ in $S^{31}$, which is of OT-FKM type, would share any common feature
with the above one of multiplicities $\{4,5\}$. We will, however, show through the classification
in this paper the striking resemblance between them.

In this paper, we will systematically employ the ideal theory,
in conjunction with
the geometry of isoparametric hypersurfaces to prove that
an isoparametric hypersurface with four principal curvatures and multiplicities
$\{4,5\}$ is the homogeneous one, and moreover, an isoparametric hypersurface with four
principal curvatures and multiplicities $\{6.9\}$ is either the homogeneous one mentioned above,
or the inhomogeneous one constructed by Ozeki and Takeuchi~\cite[I]{OT}.
Serre's criterion of normal varieties, whose technical side pertinent to our
situation we developed in~\cite{CCJ},~\cite{Chiq}, is instrumental.
It turns out the quaternion and octonion algebras also
play a fundamental role in the structures of these hypersurfaces.

The classification leaves open the only case when the multiplicity pair is
$\{7,8\}$ .

I would like to thank Josef Dorfmeister for many conversations, during our visit
to T\^{o}hoku University in the summer of 2010, to share the isoparametric
triple system approach he and Erhard Neher introduced~\cite{DN2}.

\section{Preliminaries}
\subsection{The basics} Let $M$ be an isoparametric hypersurfaces with four
principal curvatures in the sphere, and let $F$ be its Cartan-M\"{u}nzner
polynomial. To fix notation, we make the convention, by changing $F$ to $-F$
if necessary, that its two focal manifolds are $M_{+}:=F^{-1}(1)$ and
$M_{-}:=F^{-1}(-1)$ with respective codimensions $m_1+1\leq m_2+1$ in
the ambient sphere $S^{2(m_1+m_1)+1}$. The principal curvatures
of the shape operator $S_n$ of $M_{+}$ (vs. $M_{-}$) with respect to any unit
normal $n$ 
are $0,1$ and $-1$, whose multiplicities are, respectively,
$m_1,m_2$ and $m_2$ (vs. $m_2,m_1$ and $m_1$).

On the unit normal sphere bundle $UN_{+}$ of $M_{+}$, let $(x,n_0)\in UN_{+}$
be points in a small open set; here $x\in M_{+}$ and $n_0$ is normal to the tangents
of $M_{+}$ at $x$. We define a smooth orthonormal frame
$n_a,e_p,e_\alpha,e_\mu$, where $1\leq a,p\leq m_{1}$ and $1\leq\alpha,\mu\leq m_2$,
in such a way that $n_a$ are tangent to
the unit normal sphere at $n_0$, and $e_p,e_\alpha$
and $e_\mu$, respectively, are basis vectors of the eigenspaces $E_0,E_1$ and $E_{-1}$
of the shape operator $S_{n_0}$. 

\begin{convention} We will sometimes also use $b,q,\beta$ and $\nu$ in
place of
$a,p,\alpha$ and $\mu$, respectively. Henceforth, $a,p,\alpha,\mu$ are specifically
reserved for indexing the indicated normal and tangential subspaces.
\end{convention}

Each of the frame vector can be regarded as
a smooth function from $UN_{+}$ to ${\mathbb R}^{2(m_1+m_2)}$.
We have~\cite[p 14]{CCJ}, in Einstein summation convention,

\begin{eqnarray}\label{diff}
\aligned
dx= \omega^p e_p+\omega^\alpha e_\alpha+\omega^\mu e_\mu,&\qquad
dn_0=\omega^a n_a-\omega^\alpha e_\alpha+\omega^\mu e_\mu\\
dn_a=-\omega^a n_0+\theta_a^t e_t,&\qquad de_p=-\omega^p x+\theta_p^t e_t\\
de_\alpha=-\omega^\alpha x+\omega^\alpha n_0+\theta_\alpha^t e_t,&\qquad
de_\mu=-\omega^\mu x-\omega^\mu n_0+\theta_\mu^t e_t
\endaligned
\end{eqnarray}
where the index $t$ runs through the $p,\alpha$ and $\mu$ ranges, and

\begin{eqnarray}\label{theta}
\aligned
\theta_a^p=-S_{p\alpha}^a\omega^\alpha-S_{p\mu}^a\omega^\mu,
&\qquad \theta_a^\alpha=-S_{p\alpha}^a\omega^p-S_{\alpha\mu}^a\omega^\mu\\
\theta_p^\alpha=-S_{p\alpha}^a\omega^a-S_{\alpha\mu}^p\omega^\mu,
&\qquad \theta_a^\mu=-S_{p\mu}^a\omega^p-S_{\alpha\mu}^a\omega^\alpha\\
\theta_p^\mu=S_{p\mu}^a\omega^a+S_{\alpha\mu}^p\omega^\alpha,
&\qquad \theta_\alpha^\mu=(S_{\alpha\mu}^a/2)\omega^a+(S_{\alpha\mu}^p/2)\omega^p
\endaligned
\end{eqnarray}
where $S^a_{ij}:=<S(e_i,e_j),n_a>$ are the components of the second fundamental
form $S$ of $M_{+}$ at $x$, and $S^p_{\alpha\mu}$ are the $\alpha\mu$-components of $S$
at the "mirror" point $n_0\in M_{+}$ where the normal $x,e_p,1\leq p\leq m_1,$
and the tangent $n_a,1\leq a\leq 4,e_\alpha,e_\mu,1\leq\alpha,\mu\leq 5,$
form an adapted frame. Knowing $S$ at $x$ does not necessarily know $S$ at $n_0$.
This is fundamentally the reason the classification of isoparametric hypersurfaces
can be rather entangling. In any event, there are two identities connecting
$S^a_{\alpha\mu}$ and $S^p_{\alpha\mu}$ as follows~\cite[p 16]{CCJ}.

\begin{eqnarray}\label{mirror}
\aligned
&\sum_a S^a_{p\alpha}S^a_{q\beta}+\sum_a S^a_{q\alpha}S^a_{p\beta}\\
&+1/2\sum_\mu (S^p_{\alpha\mu}S^q_{\beta\mu}+S^q_{\alpha\mu}S^p_{\beta\mu})
&=\delta_{pq}\delta_{\alpha\beta}.
\endaligned
\end{eqnarray}
The other is entirely symmetric obtained by interchanging the $\alpha$ and $\mu$
ranges.

The third fundamental form of $M_{+}$ is the symmetric tensor
$$
q(X,Y,Z):=(\nabla^{\perp}_X S)(Y,Z)/3
$$
where $\nabla^{\perp}$ is the normal connection. Write $p_a(X,Y):=<S(X,Y),n_a>$
and $q^a(X,Y,Z)=<q(X,Y,Z),n_a>,0\leq a\leq m_1$.
The Cartan-M\"{u}nzner polynomial $F$ is related to $p_a$ and $q^a$ by
the expansion formula of Ozeki and Takeuchi~\cite[I, p 523]{OT}

\begin{eqnarray}\label{OT}
\aligned
&F(tx+y+w)=t^4+(2|y|^2-6|w|^2)t^2+8(\sum_{i=0}^{m}p_{i}w_{i})t\\
&+|y|^4-6|y|^2|w|^2+|w|^4-2\sum_{i=0}^{m}p_{i}^2-8\sum_{i=0}^{m}q^{i}w_{i}
\\
&+2\sum_{i,j=0}^{m}<\nabla p_{i},\nabla p_{j}>w_{i}w_{j}
\endaligned
\end{eqnarray}
where $w:=\sum_{i=0}^{m_1} w_i n_i$, $y$ is tangential to $M_{+}$ at $x$,
$p_i:=p_i(y,y)$ and $q^i:=q^i(y,y,y)$. Note that our definition of $q^i$
differs from that of Ozeki and Takeuchi~\cite[I]{OT} by a sign.

\begin{lemma}\label{Q0} $q^0(y,y,y)=-\sum_{p\alpha\mu}S^p_{\alpha\mu}X_\alpha Y_\mu Z_p$,
where $y=\sum_\alpha X_\alpha e_\alpha+\sum_\mu Y_\mu e_\mu+\sum_p Z_p e_p.$
\end{lemma}

\begin{proof} One uses~\eqref{OT} and observes that at $n_0\in M_{+}$, by~\eqref{diff},
the normal space is ${\mathbb R}x\oplus E_0$, the $0$-eigenspace of the shape operator
$S_x$ is spanned by $n_1,\cdots,n_{m_1}$, and the $\pm 1$-eigenspaces of $S_x$
are identical with $E_{1}$ and $E_{-1}$, respectively.
\end{proof}
We remark that the symmetric matrices $S_a$ of the components
$p_a,0\leq a\leq m_1,$ relative to $E_{1},E_{-1}$ and $E_0$ 
are
\begin{equation}\label{0a}
S_0=\begin{pmatrix}Id&0&0\\0&-Id&0\\0&0&0\end{pmatrix},
S_a=\begin{pmatrix}0&A_a&B_a\\A_a^{tr}&0&C_a\\B_a^{tr}&C_a^{tr}
&0\end{pmatrix},1\leq a\leq m_1,
\end{equation}
where $A_a:E_{-1}\rightarrow E_{1}$,
$B_a:E_0\rightarrow E_{1}$ and $C_a:E_0\rightarrow E_{-1}$.

\subsection{The duality between $M_{+}$ and $M_{-}$} Let $UN_{+}$
and $UN_{-}$ be respectively the unit normal bundles of $M_{+}$ and $M_{-}$.
The map
$$
(x,n_0)\rightarrow (x^*:=(x+n_0)/\sqrt{2},n_0^*:=(x-n_0)/\sqrt{2})
$$
is a diffeomorphism from $UN_{+}$ to $UN_{-}$. Finding $dx^*$ by~\eqref{diff},
we see that the normal space at $x^*$ is ${\mathbb R}n_0^*\oplus E_{+}$.
Finding $-dn_0^*$ by~\eqref{diff}, we obtain that $E_{1}^*$, the $+1$-eigenspace
of the shape operator $S_{n_0^*}$, is spanned by $n_1,\cdots,n_{m_1}$, $E_{-1}^*$,
the $-1$-eigenspace is $E_0$, and $E_0^*$, the $0$-eigenspace is $E_{-1}$.
We leave it to the reader as a simple exercise to verify the following duality property by
exploring~\eqref{diff} and~\eqref{theta} on both $M_{+}$ and $M_{-}$ at $x$ and $x^*$.

\begin{lemma}
Referring to~\eqref{0a}, let the counterpart matrices at $x^*$ and their blocks
be denoted by the same notation with an additional *.
Then
\begin{eqnarray}\label{duality}
\aligned
A_\alpha^*&=-\sqrt{2}\begin{pmatrix}S^a_{p\alpha}\end{pmatrix},1\leq\alpha\leq m_2,\\ 
B_\alpha^*&=-1/\sqrt{2}\begin{pmatrix}S^a_{\alpha\mu}\end{pmatrix},1\leq\alpha\leq m_2,\\
C_\alpha^*&=-1/\sqrt{2}\begin{pmatrix}S^p_{\alpha\mu}\end{pmatrix},1\leq\alpha\leq m_2,
\endaligned
\end{eqnarray}
where the upper scripts denote rows.
\end{lemma}

\subsection{The homogeneous example of multiplicities $\{4,5\}$}\label{scn1} Consider the
complex Lie algebra $so(5,{\mathbb C})$. The unitary group $U(5)$ acts
on it by
$$
g\cdot Z=\overline{g}Zg^{-1}
$$
for $g\in U(5)$ and $Z\in so(5,{\mathbb C})$. The principal orbits of
the action is the homogeneous 1-parameter family of isoparametric
hypersurfaces with multiplicities $(m_1,m_2)=(4,5)$.
Let the $(i,j)$-entry of $Z$ be denoted by
$a_{ij}$,
and let $a_{ij}=x_{ij}+\sqrt{-1}y_{ij}$ in which $x_{ij}$ and
$y_{ij}$ are real.
The Euclidean space is $so(5,{\mathbb C})$ coordinatized by $x_{ij}$ and
$y_{ij}$, and the Cartan-M\"{u}nzner polynomial is~\cite[II, p 27]{OT}
$$
F(Z)=-5/4\sum_{i} |Z_i|^4+3/2\sum_{i<j}|Z_i|^2|Z_j|^2-4\sum_{i<j}
|<Z_i,Z_j>|^2,
$$
where $Z_1,\cdots,Z_5$ are the row vectors of $Z$. It is readily seen that
the point $x$ with coordinates $x_{12}=x_{34}=1/\sqrt{2}$ and zero otherwise
satisfies $F(x)=1$, so that $x\in M_{+}=SU(5)/Sp(2)$. Let us introduce new coordinates
\begin{eqnarray}\nonumber
\aligned
x_{12}:=(t+w_0)/\sqrt{2},&\qquad x_{34}:=(t-w_0)/\sqrt{2},\\
x_{13}:=(w_3-z_4)/\sqrt{2},&\qquad x_{24}:=(w_3+z_4)/\sqrt{2},\\
y_{13}:=(-z_3-w_4)/\sqrt{2},&\qquad y_{24}:=(-z_3+w_4)/\sqrt{2},\\
x_{14}:=(z_2-w_1)/\sqrt{2},&\qquad x_{23}:=(z_2+w_1)/\sqrt{2},\\
y_{14}:=(w_2+z_1)/\sqrt{2},&\qquad y_{23}:=(w_2-z_1)/\sqrt{2}.
\endaligned
\end{eqnarray}
Then $w_0,\cdots,w_4$ are the normal coordinates, $z_1,\cdots,z_4$
the $E_0$-coordinates, and
\begin{eqnarray}\nonumber
\aligned
x_1:=x_{35},x_2:=y_{35},x_3&:=x_{45},x_4:=y_{45},x_5:=y_{34},\\
y_1:=x_{15},y_2:=y_{35},y_3&:=x_{25},y_4:=y_{25},y_5:=y_{12}
\endaligned
\end{eqnarray}
are the five $E_{1}$ and five $E_{-1}$ coordinates, in order.
In fact, the components of the second fundamental form of $M_{+}$
at $x$ are, by~\eqref{OT},

\begin{eqnarray}\label{2nd}
\aligned
p_0&=(x_1)^2+\cdots+(x_5)^2-(y_1)^2-\cdots-(y_5)^2,\\
p_1&=2(x_1y_1+x_2y_2+\cdots+x_4y_4)+\sqrt{2}(x_5+y_5)z_1,\\
p_2&=2(x_2y_1-x_1y_2)+2(x_3y_4-x_4y_3)+\sqrt{2}(x_5+y_5)z_2,\\
p_3&=2(x_3y_1-x_1y_3)+2(x_4y_2-x_2y_4)+\sqrt{2}(x_5+y_5)z_3,\\
p_4&=2(x_2y_3-x_3y_2)+2(x_4y_1-x_1y_4)+\sqrt{2}(x_5+y_5)z_4.
\endaligned
\end{eqnarray}
Note that the 5-by-5 matrices $A_i$ of $p_i, 1\leq i\leq 4,$ given in~\eqref{0a} are
\begin{eqnarray}\label{A1}
\aligned
A_1:=\begin{pmatrix}I&0&0\\0&I&0\\0&0&0\end{pmatrix},&\qquad
A_2:=\begin{pmatrix}J&0&0\\0&-J&0\\0&0&0\end{pmatrix},\\
A_3:=\begin{pmatrix}0&-I&0\\I&0&0\\0&0&0\end{pmatrix},&\qquad
A_4:=\begin{pmatrix}0&J&0\\J&0&0\\0&0&0\end{pmatrix},
\endaligned
\end{eqnarray}
where $I$ is the 2-by-2 identity matrix and $J$ is the 2-by-2 matrix
\begin{equation}\label{j}
J:=\begin{pmatrix}0&-1\\1&0\end{pmatrix}. 
\end{equation}
It is readily checked that the upper 4-by-4 blocks of $A_1,\cdots,A_4$,
still denoted by $A_1,\cdots.A_4$ for notational convenience, satisfy
$$
A_jA_k+A_kA_j=-2\delta_{jk}I
$$
with
\begin{equation}\label{quaternion}
A_2A_3=-A_4.
\end{equation}
Note that $A_1,\cdots,A_4$ are exactly the matrix representations of
the multiplications by $1,i,j,k$, respectively, on the right over ${\mathbb H}$.
The 5-by-4 matrices $B_i=C_i$ of $p_i,1\leq i\leq 4$, given in~\eqref{0a} are
\begin{eqnarray}\label{rank}
\aligned
B_1:=\begin{pmatrix}0&0&0&0\\1/\sqrt{2}&0&0&0\end{pmatrix},&\qquad
B_2:=\begin{pmatrix}0&0&0&0\\0&1\sqrt{2}&0&0\end{pmatrix},\\
B_3:=\begin{pmatrix}0&0&0&0\\0&0&1\sqrt{2}&0\end{pmatrix},&\qquad
B_4:=\begin{pmatrix}0&0&0&0\\0&0&0&1/\sqrt{2}\end{pmatrix},
\endaligned
\end{eqnarray}
where the first zero row in each matrix is of size 4-by-4.

Note that it follows from~\eqref{rank} that all nontrivial linear combinations
of $B_1,\cdots,B_4$ are of rank 1, which will play a decisive role later.

A calculation with the expansion formula~\eqref{OT} gives the components
of the third fundamental form $\tilde{q}$ of the homogeneous example. We will
only display $\tilde{q}^0$ for later purposes.

\begin{eqnarray}\label{q0}
\aligned
\tilde{q}^0&=-2z_4(x_1y_3+x_3y_1+x_2y_4+x_4y_2)\\
&-2z_3(-x_1y_4-x_4y_1+x_2y_3+x_3y_2)\\
&-2z_2(x_1y_1+x_2y_2-x_3y_3-x_4y_4)\\
&-2z_1(x_1y_2-x_2y_1+x_3y_4-x_4y_3)
\endaligned
\end{eqnarray}

%%\begin{eqnarray}\nonumber
%%\aligned
%%\tilde{q}^1&=z_4(-2x_1x_3-2x_2x_4+2y_1y_3+2y_2y_4)\\
%%&+z_3(2x_1x_4-2x_2x_3-2y_1y_4+2y_2y_3)\\
%%&+z_2(-x_1^2-x_2^2+x_3^2+x_4^2+y_1^2+y_2^2-y_3^2-y_4^2)\\
%%&+\sqrt{2}x_5(x_2y_1-x_1y_2+x_4y_3-x_3y_4)\\
%%&+\sqrt{2}y_5(x_1y_2-x_2y_1+x_3y_4-x_4y_3)
%%\endaligned
%%\end{eqnarray}

%%\begin{eqnarray}\nonumber
%%\aligned
%%\tilde{q}^2&=z_4(2x_1x_4-2x_2x_3+2y_1y_4-2y_2y_3)\\
%%&+z_3(2x_1x_3+2x_2x_4+2y_1y_3+2y_2y_4)\\
%%&+z_1(x_1^2+x_2^2-x_3^2-x_4^2-y_1^2-y_2^2+y_3^2+y_4^2)\\
%%&+\sqrt{2}x_5(-x_1y_1-x_2y_2+x_3y_3+x_4y_4)\\
%%&+\sqrt{2}y_5(x_1y_1+x_2y_2-x_3y_3-x_4y_4)
%%\endaligned
%%\end{eqnarray}

%%\begin{eqnarray}\nonumber
%%\aligned
%%\tilde{q}^3&=z_4(x_1^2+x_2^2-x_3^2-x_4^2+y_1^2+y_2^2-y_3^2-y_4^2)\\
%%&+z_2(-2x_1x_3-2x_2x_4-2y_1y_3-2y_2y_4)\\
%%&+z_1(-2x_1x_4+2x_2x_3+2y_1y_4-2y_2y_3)\\
%%&+\sqrt{2}x_5(-x_2y_3-x_3y_2+x_1y_4+x_4y_1)\\
%%&+\sqrt{2}y_5(-x_1y_4-x_4y_1+x_2y_3+x_3y_2)
%%\endaligned
%%\end{eqnarray}

%%\begin{eqnarray}\nonumber                                     
%%\aligned
%%\tilde{q}^4&=z_3(-x_1^2-x_2^2+x_3^2+x_4^2-y_1^2-y_2^2+y_3^2+y_4^2)\\
%%&+z_2(-2x_1x_4+2x_2x_3-2y_1y_4+2y_2y_3)\\
%%&+z_1(2x_1x_3+2x_2x_4-2y_1y_3-2y_2y_4)\\
%%&+\sqrt{2}x_5(-x_2y_4-x_3y_1-x_1y_3-x_4y_2)\\
%%&+\sqrt{2}y_5(x_1y_3+x_3y_1+x_2y_4+x_4y_2)
%%\endaligned
%%\end{eqnarray}

\subsection{The homogeneous example of multiplicities $\{6,9\}$}\label{SUB}
This is the example of OT-FKM type with multiplicity pair $(m_1,m_2)=(6,9)$
whose Clifford action
is on $M_{-}$ of codimension $9+1=10$ in $S^{31}$, given as follows.

Let $\check{J}_1,\cdots,\check{J}_8$ be the unique (up to equivalence) irreducible representation of the
(anti-symmetric)
Clifford algebra $C_8$ on ${\mathbb R}^{16}$.
Set
\begin{eqnarray}\nonumber
\aligned
P_0&:(c,d)\mapsto (c,-d),\\
P_1&:(c,d)\mapsto (d,c),\\
P_{1+i}&:(c,d)\mapsto (\check{J}_i(d),-\check{J}_i(c)),1\leq i\leq 8,
\endaligned
\end{eqnarray}
over ${\mathbb R}^{32}={\mathbb R}^{16}\oplus{\mathbb R}^{16}.$
$P_0,P_1,\cdots,P_9$ form a representation of the (symmetric) Clifford
algebra $C_{10}'$
on ${\mathbb R}^{32}$.

We know that $M_{-}$ with the Clifford action on it can be realized as the
Clifford-Stiefel manifold~\cite{FKM}. Namely,
\begin{eqnarray}\nonumber
\aligned
M_{-}=\{&(\zeta,\eta)\in S^{31}\subset{\mathbb R}^{16}\times{\mathbb R}^{16}:\\
&|\zeta|=|\eta|=1/\sqrt{2},\zeta\perp\eta,\check{J}_i(\zeta)\perp\eta,i=1,\cdots,8\}.
\endaligned
\end{eqnarray}

At $(\zeta,\eta)\in M_{-}$, the normal space is
$$
N=\text{span}<f_0:=P_0((\zeta,\eta)),\cdots,f_9:=P_9((\zeta,\eta))>.
$$
$E_0$, the 0-eigenspace of the shape operator $S_0:=S_{f_0}$, is
$$
E_0=\text{span}<g_1:=P_1P_0((\zeta,\eta)),\cdots,g_9:=P_{9}P_0((\zeta,\eta))>.
$$
$E_{\pm}$, the $\pm 1$-eigenspaces of $S_0$, are
$$
E_{\pm}:=\{X:P_0(X)=\mp X,X\perp N\}.
$$
Since $E_{+}$ (vs. $E_{-}$) consists of $(0,d)\in{\mathbb R}^{32}$ (vs.
$(e,0)\in{\mathbb R}^{32}$), we obtain
\begin{eqnarray}\label{e+}
\aligned
E_{+}&:=\{(0,d): d\perp \zeta,d\perp \eta,d\perp \check{J}_i(\zeta),\forall i\},\\
E_{-}&:=\{(e,0): e\perp \zeta,e\perp \eta,e\perp \check{J}_i(\eta),\forall i\}.
\endaligned
\end{eqnarray}
The second fundamental form $S_a:=S_{f_a}$ at $(\zeta,\eta)$ is
$$
S_{a}(X,Y)=-<P_a(X),Y>,
$$
The representation $\check{J}_1,\cdots,
\check{J}_8$ can be constructed out of the octonion algebra as follows.
Let $e_1,e_2,\cdots,e_8$ be the standard basis of the octonion algebra
${\mathbb O}$ with $e_1$ the multiplicative unit. Let $J_1,J_2,\cdots,J_7$
be the matrix representations of the octonion multiplications by
$e_2,e_2,\cdots,e_8$ on the right over ${\mathbb O}$. Then

\begin{equation}
\check{J}_i=\begin{pmatrix}J_i&0\\0&-J_i\end{pmatrix},1\leq i\leq 7,\qquad
\check{J}_8=\begin{pmatrix}0&I\\-I&0\end{pmatrix}.
\end{equation}

We may set
$$
\eta=(0,e_1/\sqrt{2}),\qquad \zeta=(e_2/\sqrt{2},0)
$$
(in fact any purely
imaginary $e$ in place of $e_2$ is fine). Then it is easily checked that
$(\zeta,\eta)\in M_{-}$. Moreover,
\begin{eqnarray}\nonumber
\aligned
E_{+}&=\{(0,d)\in{\mathbb R}^{16}\times{\mathbb R}^{16}:d=(0,\alpha)\in{\mathbb R}^8
\times{\mathbb R}^8,\alpha\perp e_1,e_2\},\\
E_{-}&=\{(e,0)\in{\mathbb R}^{16}\times{\mathbb R}^{16}:e=(\beta,0)\in{\mathbb R}^8
\times{\mathbb R}^8,\beta\perp e_1,e_2\}.
\endaligned
\end{eqnarray}
For $h_\alpha=(0,e_\alpha)\in E_{+}$ and $k_\mu=(e_\mu,0)\in E_{-},3\leq\alpha,\mu\leq 8,$ we calculate to see
\begin{eqnarray}\label{P}
\aligned
&<P_1(h_\alpha),k_\mu>=0,\qquad <P_9(h_\alpha),k_\mu>=-<e_\alpha,e_\mu>,\\
&\quad <P_{1+i}(h_\alpha),k_\mu>=0,\quad 1\leq i\leq 7.
\endaligned
\end{eqnarray}

The point is that what we are after is the second fundamental form of $M_{+}$
of codimension $6+1=7$ in $S^{31}$. Observe that
$$
((e_2,0),0)=((\zeta,\eta)+P_0((\zeta,\eta)))/\sqrt{2}\in M_{+},
$$
where by~\eqref{duality} the six 9-by-9 matrices $A_3,\cdots,A_8$ (to be compatible
with the octonion setup, we do not denote them by $A_1,A_2,\cdots,A_6$), similar to the ones in~\eqref{A1},
are given by, for $3\leq\alpha\leq 8,1\leq a,p\leq 9$,
\begin{equation}\label{A2}
A_\alpha=\begin{pmatrix}\sqrt{2}<P_a(h_\alpha),g_p>\end{pmatrix},
\end{equation}
where $A_{\alpha}$ is skew-symmetric with 
%%the $(1,j)$-entry
%%$=<e_\alpha,e_ j>$ for $2\leq j\leq 8$, 
the $(i,j)$-entry
$=<e_\alpha,e_je_i>$ for $1\leq i<j\leq 8$, and the nineth row and column
$=0$. That is, the upper 8-by-8 block of $A_\alpha$ is the matrix representation of the multiplication of $-e_\alpha$ on the right
over ${\mathbb O}$. Explicitly,

\begin{eqnarray}\nonumber
\aligned
A_3&=\begin{pmatrix}0&I&0&0&0\\-I&0&0&0&0\\0&0&0&-I&0\\0&0&I&0&0\\0&0&0&0&0\end{pmatrix},\quad
A_4=\begin{pmatrix}0&-J&0&0&0\\-J&0&0&0&0\\0&0&0&J&0\\0&0&J&0&0\\0&0&0&0&0\end{pmatrix},\\
A_5&=\begin{pmatrix}0&0&I&0&0\\0&0&0&I&0\\-I&0&0&0&0\\0&-I&0&0&0\\0&0&0&0&0\end{pmatrix},\quad
A_6=\begin{pmatrix}0&0&-J&0&0\\0&0&0&-J&0\\-J&0&0&0&0\\0&-J&0&0&0\\0&0&0&0&0\end{pmatrix},\\
A_7&=\begin{pmatrix}0&0&0&K&0\\0&0&-K&0&0\\0&K&0&0&0\\-K&0&0&0&0\\0&0&0&0&0\end{pmatrix},\quad
A_8=\begin{pmatrix}0&0&0&L&0\\0&0&-L&0&0\\0&L&0&0&0\\-L&0&0&0&0\\0&0&0&0&0\end{pmatrix},
\endaligned
\end{eqnarray}
where $J$ is given in~\eqref{j} and
\begin{equation}\label{k}
K:=\begin{pmatrix} 1&0\\0&-1\end{pmatrix},\qquad L:=\begin{pmatrix} 0&1\\1&0\end{pmatrix}.
\end{equation}
The upper 8-by-8 blocks of $A_3,\cdots,A_8$, still denoted by the same symbols
for notational convenience, satisfy
$$
A_\alpha A_\beta+A_\beta A_\alpha=-2\delta_{\alpha\beta}I;
$$
this is the unique (up to equivalence) Clifford representation of
$C_6$ on ${\mathbb R}^8$. We will employ later the five matrices
\begin{equation}\label{cliff}
\alpha_j=-A_3A_j,4\leq j\leq 8,
\end{equation}
which generate the unique (up to equivalence) representation of $C_5$ on ${\mathbb R}^8$.
Note that $I,\alpha_4,\cdots\alpha_8$ are compatible with~\eqref{A1}. Meanwhile,
$B_3,\cdots,B_8$, similar to the ones
in~\eqref{rank}, are given, in view of~\eqref{duality}, by
\begin{equation}\label{RANK}
B_\alpha=\begin{pmatrix}<P_a(h_\alpha),k_\mu>/\sqrt{2}\end{pmatrix},
1\leq a\leq 9,3\leq\alpha,\mu\leq 8,
\end{equation}
whose $(9,\alpha)$-entry is $1/\sqrt{2}$ and is zero elsewhere, in complete agreement
with~\eqref{rank}.

We remark that the third fundamental form of $M_{-}$ is
$$
<q(X,Y,Z),W>
={\mathcal C}\sum_{b=0}^{9}(<S_b(X,Y)P_b(Z),W>)/3,
$$
where ${\mathcal C}$ denotes the cyclic sum over $X,Y,Z$. In particular,
$$
q^0(h_\alpha,k_\mu,g_p)
={\mathcal C}\sum_{b=0}^9<S_b(h_\alpha,k_\mu)P_b(g_p),f_0>
=S_{b=p}(h_\alpha,k_\mu),
$$
which implies, by Lemma~\ref{Q0}, $B_\alpha=C_\alpha$ for $3\leq\alpha\leq 8$, as in the $(4,5)$ case.

\section{Normal varieties and codimension 2 estimates} This section gives a brief
account of the
background commutative algebra and algebraic geometry needed
for the subsequent development. Though we
can proceed in an algebraic way as done in~\cite{Sh}, we choose to present it in
an analytic way as done in~\cite{Gu} for more geometric intuition.

Let $V$ be an affine variety in ${\mathbb C}^n$ defined by the zeros of $m+1$ polynomials
$p_0,p_1,\cdots,p_m$, and let ${\mathcal S}$ be its singular set. A function
$f$ is weakly holomorphic in an open set $O$ of $V$ if it is holomorphic on $O\setminus{\mathcal S}$
and is locally bounded in $O$. Passing to the limit as $O$ shrinks to a point $p$,
we can talk about the germs of weakly holomorphic functions at $p$. The variety
is said to be normal at $p$ if the germs of weakly holomorphic functions at $p$ coincide
with the germs of holomorphic functions at $p$. That is, the Riemann extension
theorem holds true in the germs of neighborhoods around $p$. $V$ is said to be
normal if it is normal at all its points.

If $V$ is normal, then its irreducible components are disconnected~\cite{Gu}; or else 
a constant function with different values on different local irreducible branches,
which is not even continuous, would give
rise to a weakly holomorphic function that could be extended to a holomorphic function,
a piece of absurdity. Each irreducible component is normal whose singularity set is of codimension $\geq 2$.
The key point to this is that if we realize an irreducible normal variety $X$ of dimension $l$ locally as a finite branched covering
$\pi:X\rightarrow {\mathbb C}^l$, then the local irreducibility of $X$ gives that the branch locus $B$ of $X$ and $\pi(B)$ are both of 
dimension $l-1$, and so the singular set ${\mathcal S}\pi(B)$ of $\pi(B)$ is of codimension at least 2 in ${\mathbb C}^l$ . 
Then observe that the singular set of $X$ is contained in the preimage of ${\mathcal S}\pi(B)$. 

 In particular,
if $V$ is normal and connected, then $V$ is irreducible with the singular set
of codimension $\geq 2$.

\begin{corollary} If $p_0,p_1,\cdots,p_m$ are homogeneous polynomials whose zeros
define a normal variety $V$. Then $V$ is irreducible and the singular set of $V$
is of codimension $\geq 2$.
\end{corollary}

The corollary holds because $V$ defined by the zeros of homogeneous polynomials
is a cone, which is clearly connected.

Conversely, if $V$ is defined by the zeros of homogeneous polynomials $p_0,\cdots,p_m$, what
are the conditions that guarantee that $V$ is normal? A necessary condition
is that the singular set of $V$ is of codimension $\geq 2$.
The other crucial condition is that $p_0,\cdots,p_m$ form a {\em regular sequence}
in the polynomial ring of ${\mathbb C}^n$.

\begin{definition} %%Recall~\cite[p 152]{Ku} that 
A regular sequence in a
commutative ring $R$ with identity
is a sequence $a_1,\cdots,a_k$ in $R$ such that the ideal $(a_1,\cdots,a_k)$ is
not $R$, and moreover, $a_1$ is not a zero divisor in $R$ and $a_{i+1}$ is not
a zero divisor in the quotient ring $R/(a_1,\cdots,a_i)$ for $1\leq i\leq k-1$.
\end{definition}
We have the criterion of normality of Serre~\cite[p 457]{Ei}.

\begin{theorem}\label{serre}(Special case) Let $V\subset{\mathbb C}^n$ be defined by the zeros of homogeneous polynomials $p_0,\cdots,p_m$
that form a regular sequence in the polynomial ring of ${\mathbb C}^n$. Let $J$ be the subvariety of
of $V$ where the Jacobian matrix of $p_0,\cdots,p_m$ $< m+1$. Then $V$ is an irreducible normal variety if
the codimension of $J$ is at least 2 in $V$, in which case the ideal $(p_0,p_1,\cdots,p_m)$ is prime.
\end{theorem}

The criterion provides a scheme for checking whether a sequence of homogeneous polynomials $p_0,\cdots,p_m$ of the 
same degree $\geq 1$ in the polynomial
ring of ${\mathbb C}^n$ is a regular sequence~\cite[p 57]{CCJ}

\begin{proposition}\label{nullstellensatz} Let $p_0,\cdots,p_m$ be a sequence of linearly independent homogeneous polynomials of the same degree $\geq 1$
in the polynomial ring of ${\mathbb C}^n$. For each $0\leq k\leq m-1$, let $V_k$ be the variety defined by the zeros 
of $p_0,\cdots,p_k$ and let $J_k$ be the subvariety of $V_k$ where the Jacobian of $p_0,\cdots,p_k$ is of rank $< k+1$.
Then $p_0,p_1,\cdots,p_m$ form a regular sequence if $J_k$ is of codimension at least 2 in $V_k$ for $0\leq k\leq m-1$.
\end{proposition}

In fact, repeated applications of Theorem~\ref{serre} establish that the ideals $(p_0,\cdots,p_k)$ are all prime for $0\leq k\leq m-1$.
The linear independence of $p_0,\cdots,p_m$ of equal degree then demands that $p_{k+1}$ cannot be a zero divisor in the quotient ring $P[n]/(p_0,\cdots,p_k)$
by Nullstellensatz, where $P[n]$ stands for the polynomial ring of ${\mathbb C}^n$. The homogeneity of $p_0,\cdots,p_{m_1}$ of degree $\geq 1$
shows that $(p_0,\cdots,p_{m_1})$ is a proper ideal.

The components $p_0,\cdots,p_{m_1}$ of the second fundamental form of $M_{+}$ of an isoparametric hypersurface with four principal curvatures
are linearly independent homogeneous polynomials of second degree, which fits perfectly in Proposition~\ref{nullstellensatz}. By exploring more 
commutative algebra (the algebraic independence of a regular sequence) and investigating the codimension 2 condition in Proposition~\ref{nullstellensatz},
it is established in~\cite{Chiq} the following.

\begin{theorem}\label{THM} Let $M$ be an isoparametric hypersurface with four principal
curvatures and multiplicities $(m_1,m_2), m_1<m_2$. Assume the components
$p_0,p_1,\cdots,p_{m_1}$ of the second fundamental form of the focal
submanifold $M_{+}$ form a regular sequence in the ring of polynomials of ${\mathbb C}^{m_1+2m_2}$.
Then $M$ is of OT-FKM type.
\end{theorem}

\begin{corollary} $p_0,\cdots,p_{m_1}$ of $M_{+}$ do not form a regular
sequence in general when $(m_1,m_2)=(4,5),(3,4),(7,8)$ or $(6,9)$. 
\end{corollary}

\begin{proof} For $(m_1,m_2)=(3,4),(7,8),(6,9)$, consider an OT-FKM type
hypersurface whose
Clifford action is on $M_{-}$. If $p_0,\cdots,p_{m_1}$ formed a regular sequence,
then the isoparametric hypersurface would be of OT-FKM
type with the Clifford action on $M_{+}$; this is impossible
because such an OT-FKM type hypersurface whose Clifford action is
on $M_{-}$ is incongruent to one whose Clifford action is on $M_{+}$.
On the other hand, a hypersurface with $(m_1,m_2)=(4,5)$ can never be of OT-FKM type.
\end{proof}

It is shown in~\cite{Chiq} that $p_0,\cdots,p_{m_1}$ do form a regular
sequence when $m_2\geq 2m_1-1$ so that the isoparametric hypersurface is of OT-FKM
type. This leaves open only $(m_1,m_2)=(4,5),(3,4),(6,9)$ and $(7,8)$.
On the other hand, though $p_0,\cdots,p_3$ no longer form a regular sequence
in general for
$(m_1,m_2)=(3,4)$, an argument in~\cite{Chiq} that explores
Proposition~\ref{nullstellensatz} and the notion of Condition A~\cite[I]{OT} shows that
the isoparametric hypersurface with $(m_1,m_2)=(3,4)$ is of OT-FKM type. We will carry this scheme one step further in the next section when 
$(m_1,m_2)=(4,5)$ or $(6,9)$.

\section{The second fundamental form} We show in this section that the
second fundamental form of $M_{+}$ of
an isoparametric hypersurface with multiplicities $(m_1,m_2)=(4,5)$ in $S^{19}$ is, up to
an orthonormal frame change, identical with that of the homogeneous example given
in Section~\ref{scn1}. Furthermore, in the case $(m_1,m_2)=(6,9)$
in $S^{31}$, either
the isoparametric hypersurface is the inhomogeneous example constructed
by Ferus, Karcher and M\"{u}nzner, or, after an orthonormal frame change, the second
fundamental form of $M_{+}$ is identical with that of the homogeneous example.

Let us first recall
the codimension 2 estimates in~\cite{Chiq} that is crucial for the classification
of isoparametric hypersurfaces with four principal curvatures when the
multiplicity pair $(m_1,m_2)$ is either where $m_2\geq 2m_1-1,$ or is $(3,4)$.

Let $p_0,p_1,\cdots,p_{m_1}$ be the components of the second fundamental
form of $M_{+}$. We agree that ${\mathbb C}^{2m_2+m_1}$ consists of
points $(u,v,w)$
with coordinates $u_\alpha,v_\mu$ and $w_p$, where $1\leq\alpha,\mu\leq m_2$
and $1\leq p\leq m_1$. For $0\leq k\leq m_1$, let
$$
W_k:=\{(u,v,w)\in{\mathbb C}^{2m_2+m_1}:p_0(u,v,w)=\cdots=p_k(u,v,w)=0\}.
$$
We want to estimate the dimension of the subvariety $U_k$ of
${\mathbb C}^{2m_2+m_1}$, where
$$
U_k:=\{(u,v,w)\in{\mathbb C}^{2m_2+m_1}
:\text{rank of the Jacobian of}\; p_0,\cdots,p_k<k+1\}.
$$
$p_0,\cdots,p_{k}$ give rise to a linear system of
cones ${\mathcal C}_\lambda$ defined by
$$
c_0p_0+\cdots+c_{k}p_{k}=0
$$
with
\begin{equation}\label{lda}
\lambda:=[c_0:\cdots:c_{k}]\in{\mathbb C}P^{k}.
\end{equation}
The singular subvariety of ${\mathcal C}_\lambda$ is
$$
{\mathscr S}_\lambda:=\{(u,v,w)\in{\mathbb C}^{2m_2+m_1}:
(c_0S_{n_0}+\cdots+c_kS_{n_k})\cdot (u,v,w)^{tr}=0\},
$$
where $<S_{n_i}(X),Y>=<S(X,Y),n_i>$ is the shape operator of the focal
manifold $M_{+}$ in the normal direction $n_i$; we have 
\begin{equation}\label{union}
U_k=\cup_\lambda{\mathscr S}_\lambda.
\end{equation}
We wish to establish 
\begin{equation}\label{prime}
\dim(W_k\cap U_k)\leq\dim(W_k)-2
\end{equation}
for $k\leq m_1-1$ to verify that $p_0,p_1,\cdots,p_{m_1}$ form a regular
sequence.

We first estimate the dimension of ${\mathscr S}_\lambda$. We break it into
two cases. If $c_0,\cdots,c_{k}$ are constant multiples of either all real
or all purely imaginary numbers, then
$$
\dim({\mathscr S}_\lambda)=m_1,
$$
since $c_0S_{n_0}+\cdots+c_kS_{n_k}=cS_n$ for some unit normal vector $n$ and
some nonzero constant $c$, and we know that the
null space
of $S_n$ is of dimension $m_1$. Otherwise, after a normal basis change
we can assume that ${\mathscr S}_\lambda$ consists of elements $(u,v,w)$
of the form $(S_{n^*_1}-\tau_\lambda S_{n^*_0})\cdot(u,v,w)^{tr}=0$ for some nonzero
complex number
$\tau_\lambda$, relative to a new orthonormal normal basis
$n^*_0,n^*_1,\cdots,n^*_{k}$
in the linear span of $n_0,n_1,\cdots,n_k$. That is, in matrix form,

\begin{equation}\label{matrix}
\begin{pmatrix}0&A&B\\A^{tr}&0&C\\B^{tr}&C^{tr}&0\end{pmatrix}
\begin{pmatrix}x\\y\\z\end{pmatrix}
=\tau_\lambda\begin{pmatrix}I&0&0\\0&-I&0\\0&0&0\end{pmatrix}
\begin{pmatrix}x\\y\\z\end{pmatrix},
\end{equation}
where $x,y$ and $z$ are (complex) eigenvectors of (real) $S_{n^*_0}$ with
eigenvalues $1,-1$ and $0$, respectively.

\begin{remark}\label{rk0} We agree to choose $n_0^*$ and $n_1^*$ as follows.
Decompose $n:=c_0n_0+\cdots+c_kn_k$ into its real and imaginary parts 
$n=\alpha+\sqrt{-1}\beta$. Define $n_0^*$ and $n_1^*$ by performing the Gram-Schmidt
process on $\alpha$ and $\beta$.
\end{remark}

Lemma 49~\cite[p 64]{CCJ} ensures
that we can assume
\begin{equation}\label{BC}
B=C=\begin{pmatrix}0&0\\0&\sigma\end{pmatrix},
\end{equation}
where $\sigma$ is a nonsingular diagonal matrix of size
$r_\lambda$-by-$r_\lambda$ with $r_\lambda$ the rank of $B$,
and $A$ is of the form

\begin{equation}\label{A}
A=\begin{pmatrix}I&0\\0&\Delta\end{pmatrix},
\end{equation}
where $\Delta=\text{diag}(\Delta_1,\Delta_2,\Delta_3,\cdots)$
is of size $r_\lambda$-by-$r_\lambda$, in which $\Delta_1=0$
and $\Delta_i,i\geq 2,$ are nonzero skew-symmetric matrices expressed in the
block form
$\Delta_i=\text{diag}(\Theta_i,\Theta_i,\Theta_i,\cdots)$ with $\Theta_i$ a
2-by-2 matrix of the form
$$
\begin{pmatrix}0&f_i\\-f_i&0\end{pmatrix}
$$
for some $0<f_i<1$. We decompose $x,y,z$ into $x=(x_1,x_2),y=(y_1,y_2),
z=(z_1,z_2)$ with $x_2,y_2,z_2\in{\mathbb C}^{\,r_\lambda}$.
Equation~\eqref{matrix} is

\begin{eqnarray}\label{estimate}
\aligned
x_1=-\tau_\lambda y_1,&\qquad y_1=\tau_\lambda x_1,\\
-\Delta x_2+\sigma z_2=-\tau_\lambda y_2,&\qquad \Delta y_2+\sigma z_2=\tau_\lambda x_2,\\
\Delta(x_2&+y_2)=0.
\endaligned
\end{eqnarray}
This can be solved explicitly to obtain that 
$x_2=-y_2$ and $z_2$ can be solved (linearly) in terms of $x_2$. 
%%by the second set of
%%equations of~\eqref{estimate}.
Conversely $x_2=-y_2$
can be solved in terms of $z_2$ when $\tau_\lambda\neq\pm f_i\sqrt{-1}$ for all $i$,
so that $z$ can be chosen to be a free variable
in this case. So, either $x_1=y_1=0$, in which case
$$
\dim({\mathscr S}_\lambda)=m_1,
$$
or both $x_1$ and $y_1$ are nonzero,
in which case $y_1=\pm\sqrt{-1}x_1$ and so
\begin{equation}\label{EST}
\dim({\mathscr S}_\lambda)=m_1+m_2-r_\lambda.
\end{equation}
Since eventually we must estimate the dimension of $W_k\cap U_k$, let us cut
${\mathscr S}_\lambda$ by
$$
0=p^*_0=\sum_{\alpha}(x_\alpha)^2-\sum_{\mu}(y_\mu)^2.
$$
\noindent Case 1. $x_1$ and $y_1$ are both nonzero. This is the case of
nongeneric $\lambda\in{\mathbb C}P^k$. We substitute $y_1=\pm \sqrt{-1}x_1$ and $x_2$
and $y_2$ in terms of $z_2$ into $p^*_0=0$ to
deduce
$$
0=p^*_0=(x_1)^2+\cdots+(x_{m_2-r_\lambda})^2+z\; \text{terms};
$$
hence $p^*_0=0$ cuts ${\mathscr S}_\lambda$ to reduce the dimension by 1, i.e., by~\eqref{EST},
\begin{equation}\label{sub}
\dim(W_{k}\cap{\mathscr S}_\lambda)\leq m_1+m_2-r_\lambda-1,
\end{equation}
noting that $W_k$ is also cut out by $p^*_0,p^*_1,\cdots,p^*_k$. Meanwhile,
only a subvariety of $\lambda$ of
dimension $k-1$ in ${\mathbb C}P^k$ assumes $\tau_\lambda=\pm\sqrt{-1}$. (In fact,
the subvariety is the hyperquadric ${\mathcal Q}$. See Remark~\ref{rk1} below.)
Therefore, if we stratify ${\mathcal Q}$ into subvarieties ${\mathcal L}_j$ over which
$r_\lambda=j$, then by~\eqref{sub} an irreducible component ${\mathcal W}_j$
of $W_k\cap (\cup_{\lambda\in {\mathcal L}_j} {\mathscr S}_\lambda)$
will satisfy                            
\begin{equation}\label{W}
\dim({\mathcal W}_j)\leq\dim(W_{k}\cap{\mathscr S}_\lambda)+k-1\leq
m_1+m_2+k-2-j.
\end{equation}
\noindent Case 2. $x_1=y_1=0$. This is the case of generic $\lambda$, where
$\dim({\mathscr S}_\lambda)=m_1$, so that an irreducible component ${\mathcal V}$
of $W_k\cap (\cup_{\lambda\in{\mathcal G}}{\mathscr S}_\lambda)$,
where ${\mathcal G}$ is the Zariski open set of ${\mathbb C}P^k$ of generic
$\lambda$, will satisfy
\begin{equation}\label{v}
\dim({\mathcal V})\leq m_1+k.
\end{equation}
On the other hand, since $W_k$ is cut out by $k+1$ equations, we have
\begin{equation}\label{lower-bound}
\dim(W_k)\geq m_1+2m_2-k-1.
\end{equation}

\begin{lemma}\label{lm1} When $(m_1,m_2)=(4,5)$ $($respectively, $(m_1,m_2)=(6,9)$$)$
and $j\geq 2$, there holds in
equation~\eqref{W} the estimate
\begin{equation}\label{wk}
\dim({\mathcal W}_j)\leq \dim(W_k)-2
\end{equation}
for $k\leq m_1-1=3$ $($respectively, $k\leq 5$$)$.
\end{lemma}

\begin{proof} For~\eqref{wk} to be true, we must have both 
\begin{eqnarray}\nonumber
\aligned
m_1+m_2+k-2-j&\leq m_1+2m_2-k-3,\\
m_1+k&\leq m_1+2m_2-k-3
\endaligned
\end{eqnarray}
by~\eqref{W},~\eqref{v} and~\eqref{lower-bound}. The second inequality is $2m_2\geq 2k+3$, which is always true, while the
first is $m_2\geq 2k+1-j$, which is true if $j\geq 2$.
\end{proof}

\begin{remark}\label{rk1} 
In view of the proof of Lemma~\ref{lm1}, the codimension 2
estimate for
the case of generic $\lambda\in{\mathcal G}$ always holds true. Henceforth, we may ignore
this case and consider only the nongeneric case where $\tau_\lambda=\pm\sqrt{-1}$.

Observe that if we write $(c_0,\cdots,c_k)=\alpha+\sqrt{-1}\beta$ where
$\alpha$ and $\beta$ are
real vectors, then $\tau_\lambda=\pm\sqrt{-1}$ is equivalent to the conditions that
$<\alpha,\beta>=0$ and $|\alpha|^2=|\beta|^2$. That is, the nongeneric
$\lambda$ in~\eqref{lda}
is the hyperquadric ${\mathcal Q}$ in ${\mathbb C}P^k$.
\end{remark}

\begin{lemma}\label{LM} Suppose $(m_1,m_2)=(4,5)$ or $(6,9)$, and in the latter case
suppose the isoparametric hypersurface is not the inhomogeneous one constructed by Ferus, Karcher and M\"{u}nzner.
Then $r_\lambda\leq 1$ for all $\lambda$ in
${\mathcal Q}$.
\end{lemma}

\begin{proof} Suppose the contrary. Generic $\lambda$ in ${\mathcal Q}$
would have $r_\lambda\geq 2$.

We will only consider the $(4,5)$ case; the other case is verbatim.
The multiplicity pair $(4,5)$
cannot allow any points of Condition A on
$M_{+}$.  Hence,
one of the four pairs of
matrices $(B_1,C_1),(B_2,C_2),(B_3,C_3)$ and $(B_4,C_4)$ of the shape
operators $S_{n_1},S_{n_2},S_{n_3}$ and $S_{n_4}$, similar to the one given
in~\eqref{matrix}, must be nonzero; we may assume
one of $(B_1,C_1),(B_2,C_2)$ and $(B_3,C_3)$ is nonzero in the neighborhood of a given
point, over which generic $\lambda\in{\mathcal Q}$ have $r_\lambda\geq 2$.

Firstly, Lemma~\ref{lm1} would reduce the proof to considering
$r_{\lambda}\leq 1$.

\noindent Case 1. On ${\mathcal L}_1$ where $r_\lambda=1$: The codimension 2 estimate would still go through.
This is because~\eqref{W} is now replaced by ($j=1$)

\begin{equation}\label{WW}
\dim({\mathcal W}_j)\leq m_1+m_2+k-3-j=m_1+m_2+k-4
\end{equation}
due to the fact that such nongeneric $\lambda$ in ${\mathcal Q}$
constitute a subvariety of ${\mathcal Q}$ of dimension at most $k-2$.

\noindent Case 2. On ${\mathcal L}_0$ where $r_\lambda=0$: Now
$$
\dim({\mathcal W}_j)\leq m_1+m_2+k-3
$$
with $j=0$. We need to cut back one
more dimension to make~\eqref{WW} valid.
Since $r_\lambda=0$, we see $B^*_1=C^*_1=0$ and $A^*=I$ in~\eqref{matrix} for
$S_{n^*_1}$.
It follows that $p^*_0=0$ and $p^*_1=0$ cut ${\mathscr S}_\lambda$ in the
variety
\begin{equation}\label{free}
\{(x,\pm \sqrt{-1}x,z):\sum_\alpha (x_\alpha)^2=0\}.
\end{equation}
$(B^*_2,C^*_2)$ or $(B_3^*,C_3^*)$ of $S_{n^*_2}$ or $S_{n^*_3}$ must be
nonzero now; we may assume it is the former.
Since $z$ is a free variable in~\eqref{free},
$p^*_2=0$ will have nontrivial $z$-terms
\begin{eqnarray}\label{p2}
\aligned
0=p^*_2&=\sum_{\alpha p}S_{\alpha p}x_\alpha z_p+\sum_{\mu p}T_{\mu p}y_\mu z_p+x_\alpha y_\mu\;\text{terms}\\
&=\sum_{\alpha p}(S_{\alpha p}\pm\sqrt{-1}T_{\alpha p})x_\alpha z_p+x_\alpha x_\mu\;\text{terms},
\endaligned
\end{eqnarray}
taking $y=\pm\sqrt{-1}x$ into account, where $S_{\alpha p}:=<S(X^*_\alpha,Z^*_p),n^*_2>$
and $T_{\mu p}:=<S(Y^*_\mu,Z^*_p),n^*_2>$ are (real) entries of $B^*_2$ and $C^*_2$,
respectively, and $X^*_{\alpha},1\leq\alpha\leq m_2$,
$Y^*_{\mu},1\leq\mu\leq m_2$ and $Z^*_{p},1\leq p\leq m_1$, are orthonormal
eigenvectors for the eigenspaces of $S_{n^*_0}$ with eigenvalues $1,-1,$ and $0$,
respectively; hence the dimension
of ${\mathscr S}_\lambda$ will be cut down by 2 by $p^*_0,p^*_1,p^*_2=0$, so that again
\begin{equation}\label{EQ2}
\dim(W_{2}\cap{\mathscr S}_\lambda)\leq m_1+m_2-2,
\end{equation} 
noting that $p^*_0,p^*_1,p^*_2=0$ also cut out $W_2$.
In conclusion, we deduce
\begin{equation}\label{UB}
\dim({\mathcal W}_j)\leq\dim(W_{k}\cap{\mathscr S}_\lambda)+k-2\leq
m_1+m_2+k-4,
\end{equation}
so that the codimension 2 estimate would also go through. In conclusion,
we obtain that~\eqref{prime} holds true.

However, the validity of~\eqref{prime} would imply that the
isoparametric hypersurface is of OT-FKM type by Proposition~\ref{nullstellensatz}
and Theorem~\ref{THM}, which is absurd in the $(4,5)$ case.

In the $(6,9)$ case, the same arguments as above imply that
the isoparametric hypersurface is the inhomogeneous one constructed by Ferus, Karcher and M\"{u}nzner
since the Clifford action is on $M_{+}$,
contradicting the assumption. The lemma is proven.
\end{proof}

\begin{lemma}\label{LM1} Suppose $(m_1,m_2)=(4,5)$ or $(6,9)$, and in the latter case
suppose the isoparametric hypersurface is not the inhomogeneous one constructed by Ferus, Karcher and M\"{u}nzner.
Then $r_\lambda=1$ for generic $\lambda$ in ${\mathcal Q}$.
\end{lemma}

\begin{proof} We consider the $(4,5)$ case; the other case is verbatim.
Suppose the contrary. then $r_\lambda=0$ for all $\lambda$ in ${\mathcal Q}$.
It would follow that $B_1$ of $S_{n_1}$ is identically zero by considering
$\lambda=[1:\sqrt{-1}:0:0:0]$, because then $B_0^*$ and $B_1^*$ associated with
$S_{n_0^*}$ and $S_{n_1^*}$ are zero. Likewise, $B_a=0$ for all $1\leq a\leq 4.$
However, this would imply that the isoparametric hypersurface is of Condition A.
This is impossible.
\end{proof}

\begin{lemma} Suppose $(m_1,m_2)=(4,5)$ or $(6,9)$, and in the latter case
suppose the isoparametric hypersurface is not the inhomogeneous one constructed by Ferus, Karcher and M\"{u}nzner.
Then $r_\lambda=1$ for all $\lambda$ in ${\mathcal Q}$.
\end{lemma}

\begin{proof} For a $\lambda$ with $r_\lambda=0$ we have $A$ in~\eqref{matrix} is the
identity matrix by~\eqref{A}, so that its rank is full (=5 or 9). It follows that generic $\lambda$
in ${\mathcal Q}$ will have the same full rank property. However, for a $\lambda$
with $r_\lambda=1$, the structure of $A$ in~\eqref{A} implies that
$\Delta=0$ so that
such $A$, which are also generic, will be of rank 4 or 8. This is a contradiction.
\end{proof}

\begin{lemma}\label{lm5} Suppose $(m_1,m_2)=(4,5)$ or $(6,9)$, and in the latter case
suppose the isoparametric hypersurface is not the inhomogeneous one constructed by Ferus, Karcher and M\"{u}nzner.
Then up to an orthonormal frame change, the only nonzero row
of the $5$-by-$4$ (vs. $9$-by-$6$) matrices $B_a,1\leq a\leq 4,$ (vs. $3\leq a\leq 8$)
of $S_{n_a}$ is the last row.
\end{lemma}

\begin{proof} We will prove the $(4,5)$ case. The other case is verbatim with obvious
changes on index ranges.
For $\lambda$ in ${\mathcal Q}$, we construct $n_0^*$ and
$n_1^*$ as given in Remark~\ref{rk0} and extend them to a smooth local
orthonormal frame $n_0^*,n_1^*,\cdots,n_{m_1}^*$ such that $S_{n_0^*}$
and $S_{n_1}^*$ assume the matrix form in~\eqref{matrix},~\eqref{BC}
and~\eqref{A}. Note that $\Delta=0\,(=\Delta_1)$ in~\eqref{A} because $r_\lambda=1$;
it follows that $\sigma=1/\sqrt{2}$ in~\eqref{BC}~\cite[p 67]{CCJ}. Suppose there is
a $\lambda_0$ at which $S_{n_2^*}$ in matrix form is such that the
matrix $B_2^*$ associated with $S_{n_2^*}$ has a nonzero row other than the
last one; this property will continue to be true in a neighborhood of
$\lambda_0$. Modifying ~\eqref{free}, $p^*_0=0$ and $p^*_1=0$ now cut
${\mathscr S}_\lambda$ in the variety
\begin{equation}\label{}
\{(x_1,\cdots,x_4,\frac{t}{\sqrt{2}\tau_\lambda},
\tau_\lambda x_1,\cdots,\tau_\lambda x_4,-\frac{t}{\sqrt{2}\tau_\lambda},z_1,\cdots,z_3,t):\sum_{j=1}^4 (x_j)^2=0\}
\end{equation}
where
\begin{eqnarray}\nonumber
\aligned
x&=(x_1,x_2,x_3,x_4,x_5=t/\sqrt{2}\tau_\lambda),\\
y&=(y_1,\cdots,y_5)=(\tau_\lambda x_1,\tau_\lambda x_2,\tau_\lambda x_3,\tau_\lambda x_4,-t/\sqrt{2}\tau_\lambda)\\
z&=(z_1,z_2,z_3,z_4=t).
\endaligned
\end{eqnarray}
Meanwhile,~\eqref{p2} becomes
\begin{equation}\label{cut}
0=\sum_{\alpha=1, p=1}^{4}(S_{\alpha p}\pm\sqrt{-1}
T_{\alpha p})x_\alpha z_p\;\;(\text{mod}\;\;x_\alpha x_\mu\;\;\text{and}\;\;tz_p\;\text{terms}).
\end{equation}
The assumption that $B_2^*$ (or $C_2^*$)
assumes an
extra nonzero row other than the last one implies that one more dimension cut can be achieved
since $x_1,\cdots,x_4,z_1,\cdots,z_4$ are independent variables and~\eqref{cut}
is now nontrivial. It follows that
once more
$$
\dim(W_k\cap U_k)\leq m_1+m_2+k-4
$$
for $k\leq 3$, so that~\eqref{prime} goes through in the neighborhood
of $\lambda_0$, which is absurd as the hypersurface would be of OT-FKM type (respectively, would
be the inhomogeneous one constructed by Ferus, Karcher and M\"{u}nzner in the $(6,9)$ case).
We therefore conclude that no such $\lambda_0$ exist and so the only nonzero entry
of $B_2^*$ (or $C_2^*$) is the last one. Since any unit normal vector perpendicular
to $n_0^*$ and $n_1^*$ can be $n_2^*$, the conclusion follows.
\end{proof}

\begin{proposition}\label{prop1} Suppose $(m_1,m_2)=(4,5)$ or $(6,9)$, and in the latter case
suppose the isoparametric hypersurface is not the inhomogeneous one constructed by Ferus, Karcher and M\"{u}nzner.
Then we can choose an orthonormal frame such that the second
fundamental form of $M_{+}$ is exactly that of the homogeneous example.
\end{proposition}

\begin{proof} We will prove the $(4,5)$ case and remark on the $(6,9)$ case at the end.
$S_{n_0^*}$ is the square matrix on the right hand side
of~\eqref{matrix} while $S_{n_1^*}$ is the square one on the left hand side,
where the 1-by-1 matrix $\sigma=1/\sqrt{2}$ in~\eqref{BC} and the 1-by-1 matrix
$\Delta=0$
in~\eqref{A}. We proceed to understand $S_{n_j^*}$ with the associated
matrices $A_j,B_j$ and $C_j$ for $2\leq j\leq 4$ similar to what is given
in~\eqref{matrix}. We know by Lemma~\ref{lm5} the 5-by-4 matrices $B_j$ and
$C_j$ are of the form
$$
B_j=\begin{pmatrix}0&0\\b_j&c\end{pmatrix},\qquad
C_j=\begin{pmatrix}0&0\\e_j&f\end{pmatrix}
$$
for some real numbers $c$ and $f$. Write the 5-by-5 matrix $A_j$ as
$$
A_j=\begin{pmatrix}\alpha_j&\beta\\\gamma&\delta\end{pmatrix}
$$
with $\delta$ a real number. Then the identities~\cite[II, p 45]{OT}
\begin{eqnarray}\nonumber
\aligned
A_jA^{tr}+AA_j^{tr}+2B_jB^{tr}+2BB_j^{tr}&=0,\\
A_jA^{tr}+AA_j^{tr}+2C_jC^{tr}+2CC_j^{tr}&=0
\endaligned
\end{eqnarray}
result in

\begin{equation}\label{alpha}
\alpha_j=-\alpha_j^{tr},\qquad\gamma=0,\qquad c=f=0.
\end{equation}
On the other hand, the matrix
$$
A_jCB^{tr}+B_jC^{tr}A^{tr}+AC_jB^{tr}
$$
being skew-symmetric~\cite[II, p 45]{OT} implies
$$
\beta=0,\qquad\delta=0.
$$
Meanwhile, the identity~\cite[II, p 45]{OT}
$$
A_jA_j^{tr}+2B_jB_j^{tr}=I
$$
derives
$$
\alpha_j\alpha_j^{tr}=I,\qquad b_jb_j^{tr}=1/2.
$$
Next, the identity~\cite[II, p 45]{OT}
$$
A_jA_k^{tr}+A_kA_j^{tr}+2B_jB_k^{tr}+2B_kB_j^{tr}=0
$$
with $j\neq k$ arrives at
$$
\alpha_j\alpha_k=-\alpha_k\alpha_j,\qquad b_ib_k^{tr}=0.
$$
Lastly, the identity~\cite[II, p 45]{OT}
$$
B_j^{tr}B+B^{tr}B_j=C_j^{tr}C+C^{tr}C_j
$$
yields
$$
b_j=e_j.
$$
The upshot is that
$$
A_1=\begin{pmatrix}I&0\\0&0\end{pmatrix},
\quad A_j=\begin{pmatrix}\alpha_j&0\\0&0\end{pmatrix},j=2,3,4,
\quad B_j=C_j=\begin{pmatrix}
0&0\\b_j&0\end{pmatrix}
$$
of the same block sizes with
$$
\alpha_j\alpha_k+\alpha_k\alpha_j=-2\delta_{jk}I,\qquad <b_j,b_k>=\delta_{jk}/2.
$$
As a consequence, first of all we can perform an orthonormal basis change on
$n_2^*,n_3^*,n_4^*$ so that the resulting new $b_j$ is $1/\sqrt{2}$ at the
$j$th
slot and is zero elsewhere. Meanwhile, we can perform an orthonormal basis
change of the $E_{1}$ and $E_{-1}$ spaces
so that $I$ and $\alpha_j,2\leq j\leq 4,$ are exactly the matrix
representations
of the right multiplication of $1,i,j,k$ on 
${\mathbb H}$ without affecting the row vectors $b_j,2\leq j\leq 4$.
This is precisely the second fundamental form of the homogeneous example.

In the $(6,9)$ case, $I,\alpha_4,\cdots,\alpha_8$ can be chosen to be the ones in~\eqref{cliff}
by a frame change; multiplying them through by $A_3$ on the left, which amounts to changing
the $E_{1}$-frame, will arrive at~\eqref{A2}.
\end{proof}

\begin{corollary}\label{sp} Suppose $(m_1,m_2)=(4,5)$ or $(6,9)$, and in the latter case
suppose the isoparametric hypersurface is not the inhomogeneous one constructed by
Ferus, Karcher and M\"{u}nzner.
Then $S^p_{\alpha\mu}=0$ if $\alpha=5$ or $\mu=5$ $($respectively $\alpha=9$ or $\mu=9$$)$.
\end{corollary}

\begin{proof} Setting $\alpha=\beta=5$ or $9$ and $p=q$ in~\eqref{mirror}, the result follows
by~\eqref{rank} and~\eqref{RANK}.
\end{proof}

\section{The third fundamental form} In this section we express the third fundamental
form of an isoparametric hypersurface with multiplicities $(m_1,m_2)=(4,5)$
or $(6,9)$ in terms of $S^p_{\alpha\mu}$, provided in the latter case
the hypersurface is not the inhomogeneous one constructed by Ferus, Karcher and M\"{u}nzner. Again for simplicity in exposition,
we will only consider the $(4,5)$ case with an obvious modification for
the $(6,9)$ case.

Let us recall that if we let $S(X,Y)$ be the second fundamental form, then
the third fundamental form is $q(X,Y)Z=(\nabla^{\perp}_X)(Y,Z)/3$ with
$\nabla^{\perp}$
the normal connection. Relative to an adapted frame with the normal basis
$n_a,0\leq a\leq m_1$, and the tangential basis
$e_p,1\leq p\leq m_1,$ $e_\alpha,
1\leq\alpha\leq m_2,$ and $e_\mu,1\leq\mu\leq m_2,$
spanning $E_0,E_{+},$ and $E_{-},$ respectively, of $M_{+}$,
let $S(X,Y)=\sum_a S^a(X,Y) n_a$ and $q(X,Y,Z)=\sum_a q^a(X,Y,Z) n_a$. Then,
with the Einstein summation convention,
$$
3q^a_{ijk}\omega^k=dS^a_{ij}-\theta^a_t S^t_{ij}+\theta^t_i S^a_{tj}
+\theta^t_j S^a_{it},
$$
where $\omega^k$ are the dual forms and $\theta^s_t$ are the normal and
space connection forms. By Proposition~\ref{prop1}, choose an adapted
orthonormal frame such that~\eqref{A1} and~\eqref{rank} hold.

\begin{lemma}\label{lm6} $q^0_{ijk}=0$ when two of the three lower indexes are in the
same $p,\alpha$, or $\mu$ range.
\end{lemma}

\begin{proof} This was proved in~\cite[I, p 537]{OT}.
\end{proof}
\begin{lemma} $q^a_{pqk}=0$ for $1\leq a\leq 4$ and all $k$.
\end{lemma}

\begin{proof} $S^a_{pq}=0$ for $0\leq a\leq 4$, $S^p_{\alpha\mu}=0$
when either $\alpha=5$ or $\mu=5$ by Corollary~\ref{sp},
and $S^a_{\alpha p}=S^a_{\mu p}=0$
when $\alpha,\mu\leq 4$ by Proposition~\ref{prop1} and~\eqref{A1}. So,
in Einstein summation convention,

\begin{eqnarray}\nonumber
\aligned
3q^a_{pqk}&=3q^a_{pqk}\omega^k(e_k)\\
&=(\theta^{\alpha=5}_p S^a_{\alpha=5\; q}
+\theta^{\mu=5}_p S^a_{\mu=5\; q}+\theta^{\alpha=5}_q S^a_{p\;\alpha=5}
+\theta^{\mu=5}_q S^a_{p\; \mu=5})(e_k)=0
\endaligned
\end{eqnarray}
by~\eqref{theta} when $k$ is in either the $\alpha$ or $\mu$ range.
$q^a_{pqk}=0$ when $k$ is in the $p$-range~\cite[I, p 537]{OT}.
\end{proof}

\begin{lemma} For $1\leq \alpha,\beta\leq 4$, there holds $q^a_{\alpha\beta\mu}=0$,
while
$$
3q^a_{\alpha\beta p}=1/2\sum_{\mu=1}^4
(S^p_{\alpha\mu}S^a_{\beta\mu}+S^p_{\beta\mu}S^a_{\alpha\mu}).
$$
For $\alpha=5,$ there holds $q^a_{\alpha\beta p}=0$ while
$$
3q^a_{\alpha\beta\mu}=S^{p=a}_{\beta\mu}/\sqrt{2}.
$$
\end{lemma}

\begin{proof} For $1\leq\alpha,\beta\leq 4$, similar calculations as above
yields
$$
3q^a_{\alpha\beta p}=\theta_\alpha^\mu(e_p)S^a_{\beta\mu}+\theta_\beta^\mu(e_p)S^a_{\alpha\mu}
$$
which is the desired result by~\eqref{theta}. Likewise,
$$
3q^a_{\alpha\beta\nu}=\theta_\alpha^\mu(e_\mu)S^a_{\beta\nu}+\theta_\beta^\mu(e_\nu)S^a_{\alpha\mu}=0.
$$
For $\alpha=5$,
$$
3q^a_{\alpha\beta p}=(\theta^q_\alpha S^a_{q\beta}+\theta^q_\beta S^a_{q\alpha}
+\theta^\mu_\alpha S^a_{\mu\beta}+\theta^\mu_\beta S^a_{\mu\alpha})(e_p)=0
$$
by~\eqref{theta} and Corollary~\ref{sp}. Likewise,
$$
3q^a_{\alpha\beta\mu}=\theta^q_\beta(e_\mu)S^a_{\alpha q}=S^{p=a}_{\beta\mu}/\sqrt{2}
$$
by~\eqref{theta}, Corollary~\ref{sp} and~\eqref{rank}.
\end{proof}

A parallel argument gives the following.

\begin{lemma} For $1\leq \mu,\nu\leq 4$, there holds
$q^a_{\mu\nu\alpha}=0$,
while
$$
3q^a_{\mu\nu p}=-1/2\sum_{\alpha=1}^4
(S^p_{\alpha\mu}S^a_{\alpha\nu}+S^p_{\alpha\nu}S^a_{\alpha\mu}).
$$
For $\mu=5,$ there holds $q^a_{\mu\nu p}=0$ while
$$
3q^a_{\mu\nu\alpha}=-S^{p=a}_{\alpha\nu}/\sqrt{2}.
$$
\end{lemma}

\begin{lemma}\label{lm7} $3q^0_{p\alpha\mu}=-S^p_{\alpha\mu}$.
\end{lemma}

\begin{proof}
This is Lemma~\ref{Q0}.
%%We know $S^0_{ij}=0$ as long as $i$ and $j$ are from different
%%ranges or when they are in the $p$-range, and $=\pm \delta_{ij}$ when they are
%%in the $\alpha$ or $\mu$ range, so that
%%$$
%%3q^0_{\alpha p\mu}=3q^0_{\alpha pk}\omega^k(e_\mu)
%%=(-\theta^0_bS^b_{\alpha p}
%%+\theta^\beta_p S^0_{\alpha\beta})(e_\mu)
%%=\theta_p^\alpha(e_\mu)=-S^p_{\alpha\mu}
%%$$
%%for $\alpha\leq 4$ since $S^b_{\alpha p}=0$ by~\eqref{A1}.
%%Likewise it holds when $\mu\leq 4$.
%%Suppose $\alpha=\mu=5$. Then
%%$$
%%3q^0_{\alpha\mu p}=3q^0_{\alpha\mu k}\omega^k(e_p)
%%=(-\theta^0_b S^b_{\alpha\mu}-2\theta^\mu_{\alpha})(e_p)
%%=-2\theta^\mu_\alpha(e_p)=-S^p_{\alpha\mu}
%%$$
%%because of~\eqref{A1} and~\eqref{rank}. 
\end{proof}

\begin{lemma} For $1\leq a\leq 4$, suppose either $\alpha\leq 4$
$($respectively, $\mu\leq 4$$)$.
Then we have $q^a_{p\alpha\mu}=0$ if $p\neq a$, and
$$
3q^a_{p\alpha\mu}=\theta^5_\alpha(e_\mu)/\sqrt{2},
\;(\text{respectively}\;3q^a_{p\alpha\mu}=\theta^5_\mu(e_\alpha)/\sqrt{2})
$$
if $p=a$. Here the superscript 5 is in the $\alpha$-range
(respectively, $\mu$-range).
\end{lemma}

\begin{proof} Suppose $1\leq\alpha\leq 4$. Then
\begin{eqnarray}\nonumber
\aligned
3q^a_{\alpha p\mu}=3q^a_{\alpha pk}\omega^k(e_\mu)
&=(-\theta^a_t S^t_{\alpha p}+\theta^t_\alpha
S^a_{tp}+\theta^t_p S^a_{\alpha t})(e_\mu)\\
&=\theta^{\beta=5}_\alpha S^a_{\beta=5\; p}(e_\mu)+\theta^{\nu=5}_\alpha S^a_{\nu=5\; p}(e_\mu)
+\theta^\nu_p S^a_{\alpha\nu}(e_\mu)\\
&=\theta^\nu_p(e_\mu) S^a_{\alpha\nu}=0
\endaligned
\end{eqnarray}
if $p\neq 5$, because $S^a_{\beta=5\;p}=0$ by~\eqref{rank} and
$\theta^\nu_p(e_\mu)=0$ by~\eqref{theta}.

If $p=5$, then
$$
3q^a_{\alpha p\mu}
=(\theta^{\beta=5}_\alpha S^a_{\beta=5\; p}+\theta^\nu_p S^a_{\alpha\nu})(e_\mu).
$$
It follows that
$$
3q^a_{p\alpha\mu}=\theta^{\beta=5}_\alpha(e_\mu)/\sqrt{2}+\theta^\nu_p(e_\mu) S^a_{\alpha\nu}
=\theta^{\beta=5}_\alpha(e_\mu)/\sqrt{2}
$$
because $\theta^\nu_p(e_\mu)=0$.
\end{proof}

\begin{lemma}\label{lm9} For $\alpha=\mu=5$, we have 
$q^a_{p\alpha\mu}=0$. 
\end{lemma}

\begin{proof} We have, by~\eqref{A1}, that the fifth row and column of
$A_a$ is identically zero, so that
\begin{eqnarray}\nonumber
\aligned
3q^a_{\mu\alpha p}&=(3q^a_{\mu\alpha k}\omega^k)(e_p)=(-\theta^a_tS^t_{\mu\alpha}
+\theta^t_\mu S^a_{t\alpha}+\theta^t_\alpha S^a_{\mu t})(e_p)\\
&=(\theta^q_\mu S^a_{q\alpha}+\theta^q_\alpha S^a_{\mu q})(e_p)
=0
\endaligned
\end{eqnarray}
by~\eqref{theta}.
\end{proof}

It follows from Lemmas~\ref{lm6} through~\ref{lm9} 
that the third fundamental form $q$ of $M_{+}$
of the isoparametric hypersurface under consideration
is, for $1\leq a\leq 4$,
\begin{eqnarray}\nonumber
\aligned
q^0&:=-2\sum_{p,\alpha,\mu=1}^4 S^p_{\alpha\mu}x_\alpha y_\mu z_p\\
q^a&:=
Fz_a+\sqrt{2}(x_5-y_5)\sum_{\alpha,\mu=1}^4 S^{p=a}_{\alpha\mu}x_\alpha y_\mu  
+\sum_{p,\alpha,\beta=1}^4 U^a_{\alpha\beta p}x_\alpha x_\beta z_p\\
&+\sum_{p,\mu,\nu=1}^4 V^a_{\mu\nu p}y_\mu y_\nu z_p
\endaligned
\end{eqnarray}
where
$$
F:=\sum_{(\alpha,\mu)\neq (5,5)} f_{\alpha\mu} x_\alpha y_\mu
$$
with $f_{\alpha\mu}$ either $\sqrt{2}\theta^5_{\alpha}(e_\mu)$ or
$\sqrt{2}\theta^5_\mu(e_{\alpha})$, and 

\begin{eqnarray}
\aligned
U^a_{\alpha\beta p}&:=1/2\sum_{\mu=1}^4
(S^p_{\alpha\mu}S^a_{\beta\mu}+S^p_{\beta\mu}S^a_{\alpha\mu})\\
V^a_{\mu\nu p}&:=-1/2\sum_{\alpha=1}^4
(S^p_{\alpha\mu}S^a_{\alpha\nu}+S^p_{\alpha\nu}S^a_{\alpha\mu})
\endaligned
\end{eqnarray}
with $S^a_{\alpha\mu}$ the data in~\eqref{A1}.

\begin{lemma}\nonumber
$F=0$.
\end{lemma}
\begin{proof} $p_aq_a$ contributes
$$
f_{\alpha\mu} x_\alpha x_\beta y_\mu y_\nu z_{a},
$$
for each $1\leq a\leq 4,$ and $1\leq\beta,\nu\leq 4,$ that is not shared by
any other terms in the equation~\cite[I, p 530]{OT}
\begin{equation}\label{PQ}
p_0q^0+p_1q^1+\cdots+p_4q^4=0.
\end{equation}
\end{proof}

\section{The interplay between the second and third fundamental forms}
We show in this section that the third
fundamental form of the isoparametric hypersurface under consideration is
that of the homogeneous example for the multiplicity pair $(m_1,m_2)=(4,5)$
or $(6,9)$, provided in the latter case the hypersurface is not the inhomogeneous one constructed
by Ferus, Karcher and M\"{u}nzner. We thus arrive at the classification in these two cases.

\subsection{The $(4,5)$ case}
To set it in the intrinsic quaternionic framework, let us now identify the
normal space of $M_{+}$ spanned by $n_0,n_1,\cdots,n_4$ with
${\mathbb R}n_0\oplus {\mathbb H}$, where $n_1,\cdots,n_4$ are identified
with $1,i,j,k$, respectively. 

Then the second fundamental form in~\eqref{2nd}
can be written succinctly in the vector form
as
\begin{eqnarray}\label{pw}
\aligned
&<p,w_0n_0+W>\\
&=(|X|^2+(x_5)^2-|Y|^2-(y_5)^2)w_0+2<\overline{Y}X,W>\\
&+\sqrt{2}(x_5+y_5)<Z,W>
\endaligned
\end{eqnarray}
where
\begin{eqnarray}\nonumber
\aligned
X&:=x_1+x_2i+x_3j+x_4k,\qquad Y:=y_1+y_2i+y_3j+y_4k,\\
Z&:=z_1+z_2i+z_3j+z_4k,\qquad W:=w_1+w_2i+w_3j+w_4k
\endaligned
\end{eqnarray}
with normal coordinates $w_0,w_1,\cdots,w_4$ in the respective normal directions
$n_0,\cdots,n_4$, and $e_\alpha,e_\mu$ and $e_p$ basis vectors are also identified 
with $1,i,j,k$ in the natural way. (Recall $X,Y$ and $Z$ parametrize respectively the $E_{1},E_{-1}$ and $E_0$ spaces.)
Thus there will be no confusion to set
$$
(e_1,e_2,e_3,e_4):=(1,i,j,k)
$$
for notational convenience. Let us define
\begin{equation}\label{circ}
X\circ Y:=\sum_{p=1}^4  S^p(X,Y)\;e_p.
\end{equation}
The vector-valued third fundamental form is now

\begin{eqnarray}\label{qw}
\aligned
&<q,w_0n_0+W>\\
&=-2<X\circ Y,Z>w_0+\sqrt{2}(x_5-y_5)<X\circ Y,W>\\
&+\sum_{\mu=1}^4<X\circ e_\mu,Z><\overline{e_\mu}X,W>\\
&-\sum_{\alpha=1}^4<e_\alpha\circ Y,Z><\overline{Y}e_\alpha,W>\\
&=-2<X\circ Y,Z>w_0+\sqrt{2}(x_5-y_5)<X\circ Y,W>\\
&+<X\circ(X\overline{W}),Z>-<(YW)\circ Y,Z>
\endaligned
\end{eqnarray}
where $Xe_\mu,e_\alpha Y,X\overline{W}$ and $YW$, etc.,
are quaternionic products.

Define the $4$-by-$4$ matrices
\begin{equation}\label{T}
T^p:=\begin{pmatrix}S^p_{\alpha\mu}\end{pmatrix},\quad p=1,\cdots,4.
\end{equation}
There holds
$$
T^p_{\alpha\mu}=<e_\alpha\circ e_\mu,e_p>.
$$
We remark that in the homogeneous case these matrices are obtained by
collecting half of the coefficients, respectively, of the $z_1,\cdots,z_4$ coefficients
of $-\tilde{q}^0$ in~\eqref{q0}, which are
\begin{eqnarray}\label{T1}
\aligned
\tilde{T}^1:=\begin{pmatrix}-J&0\\0&-J\end{pmatrix},&\qquad
\tilde{T}^2:=\begin{pmatrix}I&0\\0&-I\end{pmatrix},\\
\tilde{T}^3:=\begin{pmatrix}0&J\\-J&0\end{pmatrix},&\qquad
\tilde{T}^4:=\begin{pmatrix}0&I\\I&0\end{pmatrix}.
\endaligned
\end{eqnarray}
Moreover, $T^p$ are orthogonal by~\eqref{mirror}
because $S^a_{p\alpha}=0$ for all $1\leq\alpha\leq 4$ by~\eqref{rank}.
Note that
\begin{equation}\label{useful}
<X\circ Y,e_p>=<T^p(Y),X>.
\end{equation}

\begin{lemma}
\begin{equation}\label{YZ}
<(YZ)\circ Y,Z>=0
\end{equation}
for all $Y,Z\in{\mathbb H}$.
\end{lemma}

\begin{proof} Let us set $X=x_5=0$ in~\eqref{pw} and~\eqref{qw}. Then
$$
p_0=-|Y|^2-(y_5)^2,\qquad q^0=0
$$
and for $1\leq a\leq 4$
$$
p_a=\sqrt{2}y_5<Z,e_a>,\qquad q^a=-<(Ye_a)\circ Y,Z>
$$
so that~\eqref{PQ} is
$$
0=\sum_{a=0}^4p_aq^a=-\sqrt{2}y_5<Z,e_a> <(Ye_a)\circ Y,Z>=-\sqrt{2}y_5<(YZ)\circ Y,Z>.
$$
\end{proof}

\begin{corollary}\label{qh} The matrices given in~\eqref{T} are
\begin{eqnarray}\nonumber
\aligned
T^1=\begin{pmatrix}0&a&b&c\\-a&0&-d&-e\\-b&d&0&f\\-c&e&-f&0
\end{pmatrix}
,&\qquad T^2=\begin{pmatrix}a&0&g&-h\\0&a&-i&-j\\j&-h&-f&0\\-i&-g&0&-f\end{pmatrix}\\
T^3=\begin{pmatrix}b&-g&0&k\\-j&-e&-k&0\\0&-l&b&-j\\l&0&-g&-e\end{pmatrix},
&\qquad T^4=\begin{pmatrix}c&h&-k&0\\i&d&0&-k\\-l&0&d&-h\\0&-l&-i&c\end{pmatrix}
\endaligned
\end{eqnarray}                                                                  
for some twelve unknowns $a$ to $l$.
\end{corollary}

\begin{proof}  
Polarizing~\eqref{YZ} with respect to $Y$ and $Z$, respectively,
we get
\begin{eqnarray}\label{symm}
\aligned
<(Y_1Z)\circ Y_2,Z>&=-<(Y_2Z)\circ Y_1,Z>,\\
<(YZ_1)\circ Y,Z_2>&=-<(YZ_2)\circ Y,Z_1>
\endaligned
\end{eqnarray}
Setting $Z=1$ in the first equation of~\eqref{symm}, we see $T^1_{\alpha\mu}=-S^1_{\alpha\mu}$ so that $T^1$ is
skew-symmetric. Setting $Z=i$ and let $Y_1=Y_2=1$, we obtain
$$
T^2_{21}=-T^2_{21}=0,
$$
while setting $Y_1=1,Y_2=i$ yields
$$
T^2_{22}=T^2_{11}.
$$
However, setting $Z_1=1,Z_2=i$ and $Y=1$ in the second equation of~\eqref{symm},
we see
$$
T^2_{11}=-T^1_{21}=a.
$$
Thus we get the upper left 2-by-2 block of $T^2$. Continuing this fashion finishes
the proof.
\end{proof}

\begin{corollary}\label{only} We may assume $a=f=1$ and
the only nonzero entries in the matrices in Corollary~{\rm \ref{qh}}
are $a,f,k$ and $l$.
\end{corollary}

\begin{proof} 
Recall an automorphism $\sigma$ of the quaternion algebra maps a quaternion basis
to a quaternion basis, and vice versa.

Observe that if we consider the new quaternion
basis $l_i:=\sigma(e_i),1\leq i\leq 4,$ to set
$$
X=\sigma(X'),Y=\sigma(Y'),Z=\sigma(Z'),W=\sigma(W'),
$$ 
then the second fundamental form in~\eqref{pw} remains to be of the same form
since $\sigma({\overline {Y'}}X')={\overline {\sigma(Y')}}\sigma(X')=\overline{Y}X$.
Meanwhile, by comparing the homogeneous types in~\eqref{qw}, we conclude that the circle product
$\circ$ relative to the standard quaternion basis $1,i,j,k$ is converted to 
\begin{equation}\label{ci}
X'\circ' Y'=\sigma^{-1}(\sigma(X')\circ\sigma(Y'))=\sigma^{-1}(X\circ Y)
\end{equation}
relative to the new quaternion basis $\sigma(1),\sigma(i),\sigma(j),\sigma(k)$.
Therefore, to verify the lemma, it suffices to find a quaternion basis
$l_1=e_1,l_2,l_3,l_4$ for which 

\begin{equation}\label{USEFUL}
1=<l_2\circ' l_1,l_1>=<l_2\circ' e_1,e_1>=<l_2\circ e_1,e_1>=<T^1(e_1),l_2>,
\end{equation}
where the last equality is obtained by~\eqref{useful}. It is now clear that if we define
$l_2=T^1(e_1)$, then readily~\eqref{USEFUL} is verified by the orthogonality
of $T^1$. Complete $l_1,l_2$ to a quaternion basis $l_1,\cdots, l_4$ (choose $l_3\perp l_1,l_2$
and set $l_4=l_2l_3$). Now $a=1$. It follows by the orthogonality of $T^1$ that
$b=c=d=e=0$ so that $f=\pm1$. If $f=-1$, change $l_3,l_4$ to $-l_3,-l_4$
so that we may also assume $f=1$.

It follows that $g=h=i=j=0$ by the orthogonality of $T^2$, etc.
The lemma is completed by the orthogonality of $T^p,1\leq p\leq 4$.
\end{proof}

%%Now setting $X=0$ and looking at the coefficient of $(x_5-y_5)^2$ of
%%we conclude by Lemma~\ref{same} that 
%%\begin{eqnarray}\label{yw}
%%\aligned
%%&\sum_\alpha <e_\alpha\circ Y,W_1><e_\alpha\circ Y,W_2>\\
%%&=\sum_\alpha <e_\alpha* Y,W_1><e_\alpha* Y,W_2>,
%%\endaligned
%%\end{eqnarray}
%%where
%%$$
%%\begin{eqnarray}\label{YW}
%%\aligned
%%&\sum_\alpha <e_\alpha\circ Y_1,W_1><e_\alpha\circ Y_2,W_2>\\
%%&+\sum_\alpha <e_\alpha\circ Y_2,W_1><e_\alpha\circ Y_1,W_2>\\
%%&=\sum_\alpha <e_\alpha* Y_1,W_1><e_\alpha* Y_2,W_2>\\
%%&+\sum_\alpha <e_\alpha* Y_2,W_1><e_\alpha* Y_1,W_2>.
%%\endaligned
%%\end{eqnarray}

\begin{lemma}\label {same} $<\nabla q^a,\nabla q^b>=<\nabla \tilde{q}^a,\nabla \tilde{q}^b>$
for all $1\leq a,b\leq 4$.
\end{lemma}

\begin{proof} This follows from Proposition~\ref{prop1} and the identities of
Ozeki and Takeuchi~\cite[I, p 530]{OT}
\begin{eqnarray}\nonumber
\aligned
&8<\nabla q^a,\nabla q^a>=8(<\nabla p_a,\nabla p_a>(|X|^2+|Y|^2+|Z|^2+(x_5)^2+(y_5)^2)-p_a^2)\\
&+<\nabla <\nabla p_a,\nabla p_a>,\nabla G>-24G-2\sum_{b=0}^4<\nabla p_a,\nabla p_b>^2,\;\text{and}\\
&8<\nabla q^a,\nabla q^b>=8(<\nabla p_a,\nabla p_a>(|X|^2+|Y|^2+|Z|^2+(x_5)^2+(y_5)^2)-p_ap_b)\\
&+<\nabla <\nabla p_a,\nabla p_b>,\nabla G>
-2\sum_{c=0}^4<\nabla p_a,\nabla p_c><\nabla p_b,\nabla p_c>,\quad a\neq b,
\endaligned
\end{eqnarray}
where $G=p_0^2+\cdots+p_4^2.$ Observe that the isoparametric hypersurface under
consideration and the homogeneous example have the same second fundamental form.
\end{proof}

Let us now calculate $\nabla <q,W>$
with respect to the $X,Y,Z$ (i.e., $\alpha,\mu,p$) coordinates.
%%By~\eqref{pw}
%%\begin{eqnarray}\label{npw}
%%\aligned
%%&\nabla<p,W>=2\sum_{\alpha=1}^4<\overline{Y}e_\alpha,W>e_\alpha
%%+\sqrt{2}<Z,W>\zeta_5\\
%%&+2\sum_{\mu=1}^4<\overline{e_\mu}X,W>e_\mu+\sqrt{2}<Z,W>\eta_5\\
%%&+\sqrt{2}\sum_{p=1}^4(x_5+y_5)<W,e_p>e_p,
%%\endaligned
%%\end{eqnarray}
By~\eqref{qw}
\begin{eqnarray}\label{nqw}
\aligned
&\nabla<q,W>\\
&=\sum_{\alpha=1}^4 (<e_\alpha\circ(X\overline{W}),Z>
+<X\circ(e_\alpha\overline{W}),Z>)e_\alpha\\
&+\sqrt{2}(x_5-y_5)\sum_{\alpha=1}^4<e_\alpha\circ Y,W>e_\alpha
+\sqrt{2}<X\circ Y,W>\zeta_5\\
&-\sum_{\mu=1}^4(<(e_\mu W)\circ Y,Z>-<(YW)\circ e_\mu,Z>)e_\mu\\
&+\sqrt{2}(x_5-y_5)\sum_{\mu=1}^4<X\circ e_\mu,W>e_\mu-\sqrt{2}<X\circ Y,W>\eta_5\\
&+\sum_{p=1}^4(<X\circ(X{\overline W}),e_p>-<(YW)\circ Y,e_p>)e_p,
\endaligned
\end{eqnarray}
where $\zeta_5$ and $\eta_5$ are basis vectors of $x_5$ and $y_5$, respectively.

Set 
$$
<X* Y,e_p>:=\tilde{T}^p(X,Y)
$$
with $\tilde{T}^p(e_\alpha,e_\mu)$ given in~\eqref{T1}. 

\begin{corollary}\label{decisive} $k=l=1$ in Corollary~{\rm \ref{qh}}.
\end{corollary}

\begin{proof}
%%Setting $W_1=e_1,W_2=e_3,Y_1=e_2$ and $Y_2=e_4$,
Setting $p=1,q=3,\alpha=1$ and $\beta=4$ in~\eqref{mirror} with Corollary~\ref{sp}
in mind, we obtain by the structure of $T^p$ in Corollaries~\ref{qh}
and~\ref{only} (recall $T^p_{\alpha\mu}:=S^p_{\alpha\mu}$) that
$$
kf-al=0
$$
so that $k=l$ since $a=f$. 

Setting $Z=x_5=y_5=0$ in
\begin{equation}\nonumber
<\nabla <q,W_1>,\nabla <q,W_2>>
\end{equation}
via~\eqref{nqw} and comparing homogeneous types, we obtain
\begin{eqnarray}\label{important}
\aligned
&4<X\circ Y,W_1><X\circ Y,W_2>\\
&-<X\circ (X\overline{W_1}),(YW_2)\circ Y>\\
&-<X\circ (X\overline{W_2}),(YW_1)\circ Y>\\
&=4<X* Y,W_1><X*Y,W_2>\\
&-<X* (X\overline{W_1}),(YW_2)* Y>\\
&-<X* (X\overline{W_2}),(YW_1)* Y>.
\endaligned
\end{eqnarray}
Setting $W_1=e_1$ and $W_2=e_3$, we expand the preceding identity to
derive that the $x_1^2y_2y_4$ coefficient of the
second term (on both sides) is
$$
-(T^2_{11}T^2_{44}-T^2_{11}T^2_{22})=af+a^2=2,
$$
while that of the third term
(on both sides) is
$$
T^4_{13}T^4_{24}+T^4_{13}T^4_{42}=k^2+kl=2k^2=2,
$$
so that the $x_1^2y_2y_4$ coefficient of the first term satisfy
$$
k=ak=T^1_{12}T^3_{14}=\tilde{T}^1_{12}\tilde{T}^3_{14}=1,
$$
noting that the term $T^1_{14}T^3_{12}$ in the coefficient is zero.
\end{proof}

As a consequence, we deduce that $X\circ Y=X*Y$. That is,
the third
fundamental form of the isoparametric hypersurface under consideration is
that of the homogeneous example. We conclude that the isoparametric hypersurface
is precisely the homogeneous one.

\subsection{The $(6,9)$ case} The necessary modifications are as follows.
Let $e_1,e_2,\cdots,e_8$ be the octonion basis with $e_1$ the multiplicative identity.
Then in~\eqref{pw} the positive sign in front of $2<\overline{Y}X,W>$ is changed to
the negative sign (octonion multiplication is understood now). However, changing $Z,W$ to $-Z,-W$ will convert the sign. So,
we will assume~\eqref{pw} from now on. Meanwhile,

\begin{eqnarray}\nonumber
\aligned
X&:=x_1e_1+x_2e_2+\cdots+x_8e_8,\qquad Y:=y_1e_1+y_2e_2+\cdots+y_8e_8,\\
Z&:=z_3e_3+z_4e_4+\cdots+z_8e_8,\qquad W:=w_3e_3+w_4e_4+\cdots+w_8e_8
\endaligned
\end{eqnarray}
In~\eqref{T1} for the homogeneous case, the matrices are replaced, in view of~\eqref{duality}, by
\begin{equation}\label{A3}
{\tilde T}^\mu=\begin{pmatrix}\sqrt{2}<P_a(k_\mu),g_p>\end{pmatrix},
\end{equation}
where $2\leq\mu\leq 8,$ ${\tilde T}^{\mu}$ is skew-symmetric with the
$(1,j)$-entry
$=<e_\mu,e_2e_j>$ for $2\leq j\leq 8$, the $(i,j)$-entry
$=<e_\mu,(e_2e_j)e_i>$ for $2\leq i<j\leq 8$. Explicitly,

\begin{eqnarray}\nonumber%%\label{TT}
\aligned
{\tilde T}^3&=\begin{pmatrix}0&J&0&0\\J&0&0&0\\0&0&0&J\\0&0&J&0\end{pmatrix},\quad
{\tilde T}^4=\begin{pmatrix}0&I&0&0\\-I&0&0&0\\0&0&0&I\\0&0&-I&0\end{pmatrix},\\
{\tilde T}^5&=\begin{pmatrix}0&0&J&0\\0&0&0&-J\\J&0&0&0\\0&-J&0&0\end{pmatrix},\quad
{\tilde T}^6=\begin{pmatrix}0&0&I&0\\0&0&0&-I\\-I&0&0&0\\0&I&0&0\end{pmatrix},\\
{\tilde T}^7&=\begin{pmatrix}0&0&0&L\\0&0&L&0\\0&-L&0&0\\-L&0&0&0\end{pmatrix},\quad
{\tilde T}^8=\begin{pmatrix}0&0&0&-K\\0&0&-K&0\\0&K&0&0\\K&0&0&0\end{pmatrix},
\endaligned
\end{eqnarray}
where $J,K$ and $L$ are given in~\eqref{j} and~\eqref{k}.

\begin{lemma}\label{ssymm} $T^p,3\leq p\leq 8,$ in~\eqref{T} are all
skew-symmetric. The upper left $2$-by-$2$
block of each of them is zero.  
\end{lemma}

\begin{proof} Setting $x_9=y_9=0$ in~\eqref{pw} and~\eqref{qw} (note that $x_5$ and
$y_5$ in the formulae are replaced by $x_9$ and $y_9$ in the present case),
we compare homogeneous types in~\eqref{PQ} and set $X=e_1$ to obtain
\begin{eqnarray}\nonumber
\aligned
0&=|Y|^2<e_1\circ Y,Z>-\sum_{\alpha=3}^8<\overline{Y},e_\alpha><(Ye_\alpha)\circ Y,Z>\\
&=|Y|^2<e_1\circ Y,Z>-<(Y(\overline{Y}-y_1e_1+y_2e_2))\circ Y,Z>\\
&=y_1<Y\circ Y,Z>-y_2<(Ye_2)\circ Y,Z>,
\endaligned
\end{eqnarray}
of which the coefficients of of $y_1y_iy_j$, for $3\leq i,j\leq 8,$ is
$$
0=<e_i\circ e_j+e_j\circ e_i,Z>,
$$
so that $T^p_{ij}=-T^p_{ji}$. This is also true for $(i,j)=(1,j),j\geq 3,$ or $(i,j)=(2,j),j\geq 3$.
For $(i,j)=(1,2)$, the coefficients of $(y_1)^3$ and $(y_2)^3$ result in the
$T^p_{11}=T^p_{22}=0$, while the coefficient of $(y_1)^2y_2$ gives
$$
2(T^p_{12}+T^p_{21})-T^p_{21}=0
$$
and the coefficient of $y_1(y_2)^2$ gives
$$
-T^p_{22}+T^p_{21}+T^p_{11}=0.
$$
From this we see $T^p_{12}=T^p_{21}=0$. 
\end{proof}

%%\begin{lemma}\label{QQ} $|q|^2=|\tilde{q}|^2$.
%%\end{lemma}

%%\begin{proof} This follows from the identity~\cite[I, p 530]{OT}
%%$$
%%16|q|^2=16G(|X|^2+(x_9)^2+|Y|^2+(y_9)^2+|Z|^2)-<\nabla G,\nabla G>
%%$$
%%with $G=|p|^2$ and the fact that the second fundamental form is identical with that
%%of the homogneous example.
%%\end{proof}

\begin{lemma}\label{fine} Suppose $<e_2\circ Z,Z>=0$ for all $Z\perp e_1,e_2$.
Then there is an octonion orthonormal pair of purely imaginary
vectors $(X,Y)$ in ${\mathbb O}$ such that $X,Y\perp e_2$ and
$<Y\circ X,X>\neq 0$.
\end{lemma}

\begin{proof}
Suppose the contrary. For any such pair $(X,Y)$, consider
$$
T^X:=\sum_{p=3}^8 x_p T^p:{\mathbb O}\rightarrow{\mathbb O}.
$$
%%Observe that
%%$$
%%T^X:(\text{span}<e_1,Y>)^\perp\rightarrow(\text{span}<e_1,Y>)^\perp 
%%$$
%%is orthogonal. This is because we may construct a new octonion
%%basis for which $l_1=e_1, l_2=Y, l_3=X$ and $l_4=l_2l_3$, while
%%$l_5\perp l_1,\cdots,l_4,$ and $l_6=l_2l_5,l_7=l_3l_5,l_8=l_4l_5$. 
%%Then $T^X$ is essentially the above $T^3$.
Now $<Y\circ X,X>=0$ is equivalent to $<T^X(X),Y>=0$ for all purely
imaginary $Y\perp X,e_2$, and hence in fact for all purely imaginary
$Y\perp e_2$ because
$$
<T^X(X),X>=\sum_{p=3,\alpha=3,\mu=3}^8 T^p_{\alpha\mu} x_\alpha x_\mu x_p=0
$$
by the skew-symmetry of $T^p$. Moreover, the assumption $<e_2\circ X,X>=0$
is equivalent to $<T^X(X),e_2>=0$. 
We thus conclude that $T^X(X)=\pm e_1$.
Homogenizing $<T^X(X),e_1>=\pm 1$ we obtain
$$
\sum_{p=3,\mu=3}^8 T^p_{1\,\mu}x_\mu x_p=\pm |X|^2
$$
for all purely imaginary octonion vectors $X$. Hence we conclude that
$T^p_{1p}=\pm 1$ for $3\leq p\leq 8$.
However, the first identity of~\eqref{symm} with $Z=e_p,Y_1=Y_2=e_1$ gives
$T^p_{1p}=0$, which is a contradiction.
\end{proof}

\begin{lemma}\label{zero} We may assume $T^3=\tilde{T}^3$ and
$T^4=\tilde{T}^4$.
\end{lemma}

\begin{proof} 
We first show that, in view of~\eqref{ci}, we can choose an octonion basis $l_1=e_1,l_2,\cdots,l_8$ relative to which $T^3_{41}=1$, i.e., 
\begin{equation}\label{eigen}
1=<l_4\circ l_1,l_3>=<l_2\circ l_3,l_3>,
\end{equation}
in which the second equality is obtained by the first
identity of~\eqref{symm} with
$Y_1=l_2,Z=l_3$ and $Y_2=l_1$ and the skew-symmetry of $T^p$. To this end,
note that if there is a $Z\perp e_1,e_2$ such that $<e_2\circ Z,Z>\neq 0$
we are done. For, then 
the orthogonal operator
$$
U:z\perp\,(\text{span}<e_1,e_2>)^\perp\rightarrow e_2\circ z\in
(\text{span}<e_1,e_2>)^\perp
$$
is not skew-symmetric and so the structure of an
orthogonal matrix tells us that $U$ has an eigenvector $v\perp e_1,e_2$ with
eigenvalue $\pm 1$. We may assume it is $1$ by changing $e_2$ to $-e_2$ and
construct a new octonion basis in which $l_2=-e_2,v=l_3,$ etc., so that~\eqref{eigen} holds.
Otherwise, Lemma~\ref{fine} gives rise to a pair $(X,Y)$ with $X,Y\perp e_1,e_2$.
In a similar vein to $U$, the orthogonal operator
$$
R:z\perp\,(\text{span}<e_1,Y>)^\perp\rightarrow Y\circ z\in
(\text{span}<e_1,Y>)^\perp
$$
is not skew-symmetric because $X$ is in 
$(\text{span}<e_1,Y>)^\perp$. Therefore, we can find an eigenvector $w$
with eigenvalue $1$, without loss of generality, for $R$. Construct an octonion basis
in which $l_1:=1,l_2:=Y,l_3:=w,l_4:=l_2l_3,$ etc.
This choice will leave the second fundamental form
unchanged while make $T^3_{41}=1$.

With $T^3_{41}=1$, the first identity in~\eqref{symm} with
$Z=l_3,Y_1=l_1$
and $Y_2=l_2$ gives $T^3_{32}=-1$. By skew-symmetry of $T^3$, its upper left
4-by-4 block is determined to be identical with that of $\tilde{T}^3$. The orthogonality
of $T^3$ then implies that the upper right 4-by-4 and the lower left 4-by-4
blocks of $T^3$ are zero.

Now a calculation using the first identity of~\eqref{symm} establishes that the
lower right $4$-by-$4$ block of $T^3$ is of the form
$$
\begin{pmatrix}0&-a&0&-b\\a&0&b&0\\0&-b&0&a\\b&0&-a&0\end{pmatrix}.
$$
On the other hand, setting $W_1=W_2=l_3$, the
coefficient of $(x_6)^2(y_5)^2$ of the first term on the left
in~\eqref{important} is
$$
4(T^3_{65})^2=4a^2
$$
and is $0$ on the right. The coefficient of $(x_6)^2(y_5)^2$ of
the second and third terms on the left is
$$
\sum_{i=3}^8 T^i_{68}T^i_{75}=T^4_{68}T^4_{75}=-b^2
$$
because the second identity in~\eqref{symm} derives that $T^4_{68}=-T^3_{58}=b$,
$T^4_{75}=-T^3_{85}=-b$, $T^5_{68}=T^3_{48}=0$, $T^7_{75}=T^3_{35}=0$, and
$T^8_{68}=T^3_{18}=0$; it is $-1$ on the right hand side. Therefore, we obtain
$$
4a^2-2b^2=-2,\qquad a^2+b^2=1,
$$
where the second identity is obtained by the orthogonality of $T^3$.
It follows that $a=0$ and $b=\pm 1$. We may assume $b=1$; otherwise,
changing $l_5$ to $-l_5$ does the job. In other words, $T^3=\tilde{T}^3$ now. 

That $T^4=\tilde{T}^4$ follows from the second identity of~\eqref{symm} and
that
$T^3=\tilde{T}^3$. For instance, choosing $Z_1=e_3, Z_2=e_4,Y_1=e_1$ and $Y_2=e_2$
we obtain $T^4_{42}=T^3_{32}=-1,$ etc.
\end{proof}

\begin{lemma} The upper left and lower right $4$-by-$4$ blocks of\, $T^5,T^6,T^7,T^8$
are all zero.
\end{lemma}

\begin{proof} Applying the second identity of~\eqref{symm} to $Z_1=e_5,Z_2=e_3$ and
$Y=e_1$, we obtain $T^5_{31}=T^3_{51}=0$ by Lemma~\ref{zero}. Applying
the first identity of~\eqref{symm} to $Z=e_5,Y_1=e_1,$ and $Y_2=e_7$ we see
$T^5_{57}=T^5_{31}=0$. Continuing in this fashion, we can verify that all
the upper left 4-by-4 and lower right 4-by-4 entries of $T^5$ are zero except for
$T^5_{34}=-T^5_{43}=T^5_{78}=-T^5_{87}$.

To show $T^5_{43}=0$, we let $p=5, q=3,\alpha=4$ and $\beta=2$ in~\eqref{mirror}.
The matrix entries in~\eqref{A2} give $S^a_{\alpha\mu}=A^a_{\alpha\mu}$,
and recall we set
$S^p_{\alpha\mu}=T^p_{\alpha\mu}$. We derive
$$
T^5_{43}=T^5_{43}T^3_{23}+T^3_{41}T^5_{21}=-2\sum_a (A^a_{54}A^a_{32}+A^a_{34}A^a_{52})=0.
$$
 
 The same goes through for $T^6,T^7,T^8$ with $p$ replaced by $6,7,8$.
\end{proof}

\begin{lemma} The lower left $4$-by-$4$ blocks of $T^5,T^6,T^7,T^8$ are
\begin{eqnarray}\nonumber
\aligned
T^5:\begin{pmatrix}0&-a&-b&-c\\a&0&-d&-e\\b&d&0&-f\\c&e&f&0\end{pmatrix},
&\qquad T^6:\begin{pmatrix}-a&0&j&-i\\0&-a&-h&g\\g&i&k&0\\h&j&0&k\end{pmatrix},\\
T^7:\begin{pmatrix}-b&-j&0&m\\-g&e&m&0\\0&l&-b&g\\l&0&j&e\end{pmatrix},
&\qquad T^8:\begin{pmatrix}-c&i&-m&0\\-h&-d&0&m\\-l&0&-d&i\\0&l&-h&-c\end{pmatrix}
\endaligned
\end{eqnarray}
a priori for some thirteen unknowns $a$ through $m$. 
\end{lemma}

\begin{proof} Assuming the unknowns $a$ through $f$ for the lower
triangular block of the lower 4-by-4 block of $T^5$ and setting
$T^6_{71}:=g, T^6_{81}:=h, T^6_{72}:=i,T^6_{82}:=j,T^6_{73}=k,T^7_{81}=l$
and $T^7_{54}=m$, one uses the two identities in~\eqref{symm} repeatedly to obtain
all other entries in terms of these thirteen unknowns.
\end{proof}

\begin{lemma} The only nonzero entries in the above matrices are
$a,f,k,l,m$ of magnitude $1$ 
with the property that $a=-f=k$ and $l=m$. 
\end{lemma}

\begin{proof} We know $T^iT^j=-T^jT^i$ when $i\neq j$ 
by~\eqref{mirror},~\eqref{RANK}, Corollary~\ref{sp} and Lemma~\ref{ssymm}.

Now, $(i,j)=(3,5)$
or $(4,5)$ gives $a=-f$ and $b=c=d=e=0$. $(i,j)=(3,6)$ or $(4,6)$ gives
$a=k$ and $g=h=i=j=0$. Lastly, $(i,j)=(3,7)$ gives $l=m$.
\end{proof}

\begin{corollary} $a=1$ and $l=-1$. In particular, $T^5=\tilde{T}^5,T^6=\tilde{T}^6,
T^7=\tilde{T}^7,T^8=\tilde{T}^8$.
\end{corollary}

\begin{proof} The proof is similar to the one in Corollary~\ref{decisive}.
Choosing $W_1=e_5$ and $W_2=e_3$, the $x_4x_6(y_1)^2$ coefficient of
the second term (on both sides) is
$$
T^4_{68}T^4_{31}-T^4_{42}T^4_{31}=-2,
$$
while that of the third term (on both sides) is
$$
-T^6_{62}T^6_{51}-T^6_{48}T^6_{51}=-a^2-ka=-2.
$$
Therefore, the $x_4x_6(y_1)^2$ coefficient of the first term satisfies
$$
a=T^5_{61}T^3_{41}=\tilde{T}^5_{61}\tilde{T}^3_{41}=1,
$$
so that $k=1$. In particular, $T^5=\tilde{T}^5$ and $T^6=\tilde{T}^6$.
%%Choosing $W_1=e_6$ and $W_2=e_4$,
%%the $x_1x_7(y_3)^2$ coefficient of the second
%%term (on both sides) is
%%$$
%%T^3_{76}T^3_{23}+T^3_{14}T^3_{23}=-2,
%%$$
%%while that of the third term (on both sides) is
%%$$
%%T^5_{74}T^5_{83}-T^5_{16}T^5_{83}=-2.
%%$$
%%Therefore, the $x_1x_7(y_3)^2$ coefficient of the first term satisfies
%%$$
%%k=T^6_{73}T^4_{13}=\tilde{T}^6_{73}\tilde{T}^4_{13}=1.
%%$$
%%In particular, $T^6=\tilde{T}^6.$
Choosing $W_1=e_7$ and $W_2=e_3$,
the $x_1x_5(y_4)^2$ coefficient of the second term (on both sides) is
$$
-T^4_{57}T^4_{24}-T^4_{13}T^4_{24}=-2,
$$
while that of the third term (on both sides) is
$$
T^8_{53}T^8_{64}-T^8_{17}T^8_{64}=-m^2-lm=-2.
$$
Therefore, the $x_1x_5(y_4)^2$ coefficient of the first term is
$$
-m=T^7_{54}T^3_{14}=\tilde{T}^7_{54}\tilde{T}^3_{14}=1.
$$
In particular, $T^7=\tilde{T}^7$. It follows that $T^8=\tilde{T}^8$.
\end{proof}
As a consequence, the isoparametric hypersurface is precisely the homogeneous one.

\end{document}